\documentclass[12pt]{article} 
\usepackage[sectionbib]{natbib}
\usepackage{array,epsfig,fancyheadings,rotating}
\usepackage{hyperref}
\usepackage{indentfirst}
\usepackage{sectsty, secdot}
\sectionfont{\fontsize{12}{14pt plus.8pt minus .6pt}\selectfont}
\renewcommand{\theequation}{\thesection\arabic{equation}}
\subsectionfont{\fontsize{12}{14pt plus.8pt minus .6pt}\selectfont}

\usepackage{comment}
\usepackage{amsmath}
\usepackage{amssymb}
\usepackage{amsfonts}
\usepackage{multirow}
\usepackage{amsthm}

\newtheorem{theorem}{Theorem}

\newtheorem{corollary}{Corollary}

\usepackage{dcolumn}
\usepackage{amsthm,epsfig}
\usepackage{array,dcolumn,graphicx,amsmath}
\usepackage{booktabs,lscape,multirow,psfrag,rotating}
\usepackage{amssymb,amsmath,amsthm, color, bm}
\usepackage{longtable}
\usepackage{booktabs,lscape,multirow,rotating}
\usepackage[normal]{caption2}
\usepackage[mathscr]{eucal}
\usepackage{eucal}
\usepackage{amssymb,pstricks}

\usepackage{setspace} 
\newcommand{\mXbar}{\overline{\bf{X}}}

\newcommand{\bbeta}{\boldmath$\beta$}
\newcommand{\mbbeta}{\mbox{\bbeta}}
\newcommand{\bSigma}{\boldmath$\Sigma$}
\newcommand{\mbSigma}{\mbox{\bSigma}}

\newcommand{\bLambda}{\boldmath$\Lambda$}
\newcommand{\mbLambda}{\mbox{\bLambda}}
\newcommand{\bGamma}{\boldmath$\Gamma$}
\newcommand{\mbGamma}{\mbox{\bGamma}}
\newcommand{\bdelta}{\boldmath$\delta$}
\newcommand{\mbdelta}{\mbox{\bdelta}}

\newcommand{\btheta}{\boldmath$\theta$}
\newcommand{\mbtheta}{\mbox{\btheta}}
\newcommand{\bmu}{\boldmath$\mu$}
\newcommand{\mbmu}{\mbox{\bmu}}
\newcommand{\btau}{\boldmath$\tau$}
\newcommand{\mbtau}{\mbox{\btau}}

\newcommand{\bPhi}{\boldmath$\Phi$}
\newcommand{\mbPhi}{\mbox{\bPhi}}
\newcommand{\bPsi}{\boldmath$\Psi$}
\newcommand{\mbPsi}{\mbox{\bPsi}}

\newcolumntype{d}[1]{D{.}{.}{#1}}
\usepackage{comment}

\usepackage{lscape} 
\usepackage{threeparttable}
\usepackage{subfigure}
\usepackage{multirow}
\usepackage{array}
\usepackage{tabularx}
\usepackage{ltablex}

\usepackage{amsmath}               
{
	\theoremstyle{plain}
	
}

\usepackage{titlesec} % Changing section heading font size
\titleformat{\section}{\Large\bfseries}{\thesection}{1em}{}
\titleformat{\subsection}{\normalsize\bfseries}{\thesubsection}{1em}{}
\titleformat{\subsubsection}{\normalsize\bfseries}{\thesubsubsection}{1em}{}

\usepackage{float}
\usepackage{graphicx}
\graphicspath{ {./figures/} }

\usepackage{xcolor}
\usepackage{hyperref}
\hypersetup{
	colorlinks,
	citecolor=violet,
	linkcolor=blue,
	urlcolor=blue}

\usepackage{booktabs}
\usepackage{multirow}
\usepackage{siunitx}

\numberwithin{equation}{section}

\begin{document}
	\renewcommand{\baselinestretch}{2}
	
	\markright{ \hbox{\footnotesize\rm Statistica Sinica
			%{\footnotesize\bf 24} (201?), 000-000
		}\hfill\\[-13pt]
		\hbox{\footnotesize\rm
			%\href{http://dx.doi.org/10.5705/ss.20??.???}{doi:http://dx.doi.org/10.5705/ss.20??.???}
		}\hfill }
	
	\markboth{\hfill{\footnotesize\rm FIRSTNAME1 LASTNAME1 AND FIRSTNAME2 LASTNAME2} \hfill}
	{\hfill {\footnotesize\rm FILL IN A SHORT RUNNING TITLE} \hfill}
	
	\renewcommand{\thefootnote}{}
	$\ $\par
	%%%%%%%%%%%%%%%%%%%%%%5
	
	\fontsize{12}{14pt plus.8pt minus .6pt}\selectfont \vspace{0.8pc}
	\centerline{\large\bf The decomposite $T^{2}$-test when the dimension is large}
	\vspace{2pt} 
	%\centerline{\large\bf HERE IF A SECOND LINE IS NEEDED}
	%\vspace{.4cm} 
	\centerline{Chia-Hsuan Tsai~~and~~Ming-Tien Tsai} 
	\vspace{.4cm} 
	\centerline{\it Institute of Statistical Science, Academia Sinica, Taipei.}
	\vspace{.55cm} \fontsize{9}{11.5pt plus.8pt minus.6pt}\selectfont
	%%%%%%%%%%%%%%%%%%%
	
	\begin{quotation}
		\noindent {\it Abstract:}
		In this paper, we discuss tests for mean vector of high-dimensional data when the dimension $p$
		is a function of sample size $n$. One of the tests, called the decomposite $T^{2}$-test, 
		in the high-dimensional testing problem is constructed based on the estimation 
		work of Ledoit and Wolf (2018), which is an optimal orthogonally equivariant estimator 
		of the inverse of population covariance matrix under Stein loss function.
		The asymptotic distribution function of the test statistic is investigated under a sequence of 
		local alternatives. The asymptotic relative efficiency is used to see whether a test is optimal 
		and to perform the power comparisons of tests. 
		An application of the decomposite $T^{2}$-test is in testing significance 
		for the effect of monthly unlimited transport policy on public transportation, 
		in which the data are taken from Taipei Metro System.
		
		\vspace{9pt}
		\par
	\end{quotation}\par

	\def\thefigure{\arabic{figure}}
	\def\thetable{\arabic{table}}
	
	\renewcommand{\theequation}{\thesection.\arabic{equation}}

	\fontsize{12}{14pt plus.8pt minus .6pt}\selectfont

	\section{Introduction}
	\indent Let ${\bf X}_{i}, i=1, \ldots, n$, be $n$ $i.i.d.$ random vectors having a $p$-dimensional 
	multinormal distribution with mean vector ${\mbmu}$ and unknown positive definite 
	covariance matrix ${\mbSigma}$. In this paper, we are interested in testing the hypothesis
	\vskip-1.5cm
	\begin{align}
		H_{0}:{\mbmu}={\bf 0}~\mbox{versus}~H_{1}:{\mbmu}\ne {\bf 0}, ~\label{eq:1.1}
	\end{align}
	\vskip-0.5cm
	\noindent when both dimension $p$ and sample size $n$ are large. Let
	\vskip-1.5cm
	\begin{align}
		{\mXbar}=\frac{1}{n}\sum_{i=1}^{n}{\bf X}_{i} ~\mbox{and}~{\bf S}=\frac{1}{n-1}\sum_{i=1}^{n}
		({\bf X}_{i}-{\mXbar})({\bf X}_{i}-{\mXbar})^{\top}. ~\label{eq:1.2}
	\end{align}
	\vskip-0.5cm
	\noindent The Hotelling's $T^{2}$-test statistic is given by
	\vskip-1.5cm
	\begin{align}
		T^{2}=n{\mXbar}^{\top}{\bf S}^{-1}{\mXbar}. ~\label{eq:1.3}
	\end{align}
	\vskip-0.5cm
	\indent When the dimension $p ~(< n)$ is fixed, the well-known Hotelling's $T^{2}$-test 
	enjoys many optimal properties (Anderson [1]). However, when the dimension $p$ becomes large,
	the sample covariance matrix ${\bf S}$ may not be a consistent estimator of 
	population covariance matrix $\mbSigma$ when $p \geq n$. 
	Such situation makes it a hard work to estimate the precision matrix and to 
	make further usage of it. Dempster [8] [9], Bai and Saranadasa [2] first observed this 
	phenomenon, proposed a non-exact test for the hypothesis testing problem 
	(1.1) with the dimension larger than the sample size. 
	Three decades later, Bai and Saranadasa [2] proposed a new test, 
	which ignored the information of ${\bf S}$ by taking identity matrix for replacement to simplify 
	the problem. The result showed that their test has the same asymptotic power as that of 
	the Dempster's test under some assumptions on the dimension, mean vector ${\mbmu}$ 
	and population covariance matrix ${\mbSigma}$. 
	Along this line, Chen and Qin [6], further modified the test statistic of 
	Bai and Saranadasa [2]. Srivastava and Du [23] and Srivastava [24]
	used the partial information of ${\bf S}$, namely the diagonal elements, to construct new test statistic. 
	Later, Park and Ayyala [21], modified the Srivastava type test
	by incorporating some information of correlations. Feng et al. [10] assumed that 
	the matrix ${\bf S}$ has a kind of block diagonal structure to construct the composite Hotelling's $T^{2}$ 
	type test statistic. On the other hand, Chen et al. [5], used the quantity 
	${\bf S}+\lambda {\bf I}$ to replace ${\bf S}$ in (1.3), where $\lambda > 0$. 
	They used the notion of ridge regression which is highly related
	to the concept of rotation-equivariant property after matrix decomposition. Then by using 
	method of shrinkage estimation they constructed the regularized Hotelling's test statistic
	and studied its asymptotic distribution. All the tests mentioned
	above can be viewed as various versions of regularized Hotelling's $T^{2}$-test. 
	Most of the situations considered are under the setup that both dimension $p$ and 
	sample size $n$ are large so that $\lim_{n \to \infty} p/n=c, c \in (0, \infty)$.
	
	\indent In this paper, we concentrate on the situation that $c \in (0, 1)$. Different from 
	those approaches existing in the literature, our approach try to reveal more information of 
	correlations in terms of eigenvalues. Stein [26] proposed the orthogonally equivariant estimator 
	of covariance matrix and Ledoit and Wolf [14] proposed another orthogonally equivariant estimator 
	of inverse covariance matrix. Ledoit and Wolf [16] claimed that their estimator is asymptotically optimal 
	in the sense of minimizing the Stein loss.
	
	\indent The rest of the paper is organized as follows. The notion of orthogonally equivariant estimators of 
	covariance matrix for large dimensional situation and some simple notations of random matrix theory 
	are introduced in Section~2. The decomposite $T^{2}$-test statistic is presented in Section~3. 
	And the asymptotically equivalent statistic $T^{2}_{0}$ along with its asymptotic local power property 
	are also investigated in the same section. Power comparisons based on the asymptotic 
	relative efficiency are discussed in Section~4. A real example is analyzed via the bootstrap test based on the decomposite 
	$T^{2}$-test statistic and the Hotelling's $T^{2}$-test statistic, respectively in Section~5. 
	The conclusion is given in Section~6.
	
	%%%%%%%%%%%%%%%%%%%%%
	%Section_2
	%%%%%%%%%%%%%%%%%%%%%
	\section{The orthogonally equivariant estimators}
	\indent The class of orthogonally equivariant estimators of covariance matrix is constituted of all the estimators
	having the same eigenvectors as the sample covariance matrix. Consider the sample spectral decomposition, i.e.,
	${\bf S}={\bf U}{\mbLambda}{\bf U}^\top$, where ${\mbLambda}$ is a diagonal matrix with eigenvalues
	$\lambda_{i, p}, i = 1, \ldots, p$, and ${\bf U}=({\bf u}_{1}, \ldots, {\bf u}_{p})$ is the corresponding orthogonal
	matrix with ${\bf u}_{i}$ being the corresponding eigenvector with respective to $\lambda_{i, p}, i = 1, \ldots, p$.
	Similarly, for the spectral decomposition of population covariance matrix, we have 
	${\mbSigma} = {\bf V}{\mbGamma}{\bf V}^\top$,
	where ${\mbGamma}$ is a diagonal matrix with eigenvalues $\gamma_{i, p}, i = 1, \ldots, p$, and ${\bf V}$ is the
	corresponding orthogonal matrix. With respect to the Stein loss function, Stein [26], [27] 
	considered the orthogonally equivariant nonlinear shrinkage estimator which of the form
	\vskip-1.5cm
	\begin{align}\label{eq:2.1}
		&\widehat{\mbSigma}_{S}={\bf U}\widehat{\mbPhi}({\mbLambda}){\bf U}^{\top}, \mbox{where}~\widehat{\mbPhi}({\mbLambda})
		= \mbox{diag}(\widehat{\phi}_{1,p}({\mbLambda}), \ldots, \widehat{\phi}_{p, p}({\mbLambda}))
		~\mbox{with}~  \\[1ex] \nonumber
		&\widehat{\phi}_{i,p}({\mbLambda}) = n\lambda_{i,p}\left(n-p+1-2\lambda_{i,p}\sum_{j \ne i}
		\dfrac{1}{\lambda_{j,p}-\lambda_{i,p}}\right)^{-1}, ~i=1, \ldots, p.
	\end{align}
	However, some of the $1/(\lambda_{i,p}-\lambda_{j,p})$ might be negative and non-monotone. 
	To mitigate the problems, Stein recommended to use an isotonizing algorithm procedure to 
	adjust his estimators in (2.1). Stein's isotonized estimator has been considered as a gold standard, 
	thereafter a large strand of literature on orthogonally equivariant estimation of 
	population covariance matrix was generated.
	
	\indent The same as Ledoit and P{\'e}ch{\'e} [12], we make the following assumptions:
	
	\indent A1. Let ${\bf X}_{i} = \mbSigma^{1/2}{\bf y}_{i}, i=1,\ldots,n$, where ${\bf y}_{i}$ is a 
	$p \times 1$ vector of independent and identically distributed random variables $y_{ij}$. 
	Each $y_{ij}$ has mean $0$, unit variance and 12th absolute central moment 
	bounded by a constant.
	
	\indent A2. For large $(n,p)$ setup, the large dimensional asymptotic framework is setup when $(n,p) \to \infty$
	such that $p/n \to c$ is fixed, $0 \leq c < 1$ in this paper.
	
	\indent A3. The population covariance matrix is positive definite matrix. 
	Furthermore, $\|\mbSigma\|=O(p^{1/2})$, where
	$\|\cdot\|$ is the $L_{2}$ norm of a matrix.
	
	\indent A4. Let $0 < \gamma_{1, p} \leq \cdots \leq \gamma_{p, p} $. 
	The empirical spectral distribution of $\mbSigma$
	defined by $H_{n}(\gamma)=\frac{1}{p}\sum_{i=1}^{p}1_{[\gamma_{i, p}, \infty)}(\gamma)$, 
	converges as $p \to \infty$ to a
	probability distribution function $H(\gamma)$ at every point of continuity of $H$. 
	The support of $H$, $\mbox{supp}(H)$,
	is included in a compact set $[h_1, h_2]$ with $0 < h_1 \leq h_2 < \infty$.
	
	\indent Let $F_{n}(\lambda)=\frac{1}{p}\sum_{i=1}^{p}1_{[\lambda_{i, p}, \infty)}(\lambda)$ 
	be the sample spectral distribution and $F$ be its limiting distribution. We also assume that 
	there exists a nonrandom real function $\phi(x)$ defined on the support of $F$ and is 
	continuously differentiable on the support.\\
	\indent The Stieltjes transform of distribution function $F$ is defined by
	\vskip-1.3cm
	\begin{align}\label{eq:2.2}
		m_{F}(z) = \int_{-\infty}^{\infty}
		\dfrac{1}{\lambda-z}\, dF(\lambda), \linebreak \forall z \in {\mathcal C}^{+},
	\end{align}
	\vskip-0.5cm
	\noindent where ${\mathcal C}^{+}$ is the half-plane of complex numbers with a strictly positive imaginary part. 
	The empirical version is 
	\vskip-1.5cm
	\begin{align}\label{eq:2.3}
		m_{F_{n}}(z)=\dfrac{1}{p}\sum_{i=1}^{p}\dfrac{1}{\lambda_{i,p}-z}.
	\end{align}	
	\vskip-0.5cm
	\noindent Choi and Silverstein [7] showed that 
	\vskip-1.8cm
	\begin{align}\label{eq:2.4}
		\mathop{lim} \limits_{z\in {\mathcal C}^{+} \to x \in \mathcal R-\left\{0\right\} }m_{F}(z)= \check{m}_{F}(x)
	\end{align}	
	\vskip-0.5cm
	\noindent exists. Subsequently, the well known Mar{\v c}enko and Pastur equation (Choi and Silverstein [7]) in literature can be expressed in the form
	\vskip-1.5cm
	\begin{align}\label{eq:2.5}
		m_{F}(z)=\int_{-\infty}^{\infty}\frac{1}{\gamma[1-c-czm_{F}(z)]-z}dH(\gamma), \forall z \in  {\mathcal C}^{+},
	\end{align}
	\vskip-0.5cm
	\noindent where $H$ denotes the limiting behavior of the population spectral distribution. 
	Upon the Mar{\v c}enko-Pastur equation, meaningful information of the population spectral distribution can be 
	retrieved under the large dimensional asymptotic framework. Ledoit and P{\'e}ch{\'e} [12] extended the results to 
	the more general situations including the case of the precision matrix ${\mbSigma^{-1}}$. 
	In addition to estimate the population covariance matrix ${\mbSigma}$, they also estimate the 
	inverse of population covariance matrix $\mbSigma^{-1}$.
	
	\indent Consider $\Theta^{g}_{n}(z)=p^{-1}{\mbox tr}[({\bf S}-z{\bf I})^{-1}g(\mbSigma)]$, where $g(\cdot)$ is a scale function on the eigenvalues of a matrix such that $g(\mbSigma)={\bf V}\mbox{diag}(g(\gamma_{1, p}), \ldots,
	g(\gamma_{p, p})){\bf V}^{\top}$ (Ledoit and P{\'e}ch{\'e} [12], page 236). Ledoit and P{\'e}ch{\'e} [12] proved that $\Theta^{g}_{n}(z)$ converges to $\Theta^{g}(z)$ almost surely under the conditions A1-A4, where
	\vskip-1.2cm
	\begin{align}\label{eq:2.6}%%p.6
		\Theta^{g}(z)=\int_{-\infty}^{\infty}\frac{1}{\gamma[1-c-czm_{F}(z)]-z}
		g(\gamma)dH(\gamma), \forall z \in  {\mathcal C}^{+}.  
	\end{align}
	\vskip-0.3cm
	\noindent Note that if $g(\mbSigma)={\bf I}$, then the equation (2.6) reduces to the equation (2.5).
	
	\indent Ledoit and Wolf [14] suggested to use the oracle estimators
	${\bf P}^{or}={\bf U}{\bf A}^{or}{\bf U}^{\top}$ for ${\mbSigma}^{-1}, \mbox{where}~{\bf A}^{or} =\mbox{diag}({a}^{or}_{1}, \ldots {a}^{or}_{p})$, with
	\vskip-1.3cm
	\begin{align}~\label{eq:2.7}
		a^{or}_{i}=\lambda^{-1}_{i}\left(1-c-2{c}\lambda_{i}\mbox{Re}\left[\check{m}_{F}(\lambda_{i})\right]\right), c \in [0, 1), \forall i=1,
		\ldots, p. 
	\end{align}
	\vskip-0.3cm
	\noindent Note that $\lambda_{i}$ is the quantile, i.e., $F^{-1}(\alpha)=\lambda_{i}$ with $[p\alpha]=i, 0 < \alpha < 1$, $i=1, \ldots, p$, where $[x]$ denotes the largest integer of $x$. 
	Also note that $a^{or}_{i}$ is nonrandom and an estimable quantity, $\forall i=1,\ldots,p$. 
	Let ${\bf u}^{\top}_{i}{\mbSigma}^{-1}{\bf u}_{i}=a^{*}_{i}$, Ledoit and P{\'e}ch{\'e} [12]
	showed that $a^{*}_{i}$ is approximated by $a^{or}_{i}, i=1, \ldots, p$. 
	Slightly different from the Stein's estimator in (2.1), 
	Ledoit and Wolf [14] proposed the estimator of ${\mbSigma}$ which is of the form
	\vskip-1.3cm
	\begin{align}~\label{eq:2.8}
		&\widehat{\mbSigma}_{LW}={\bf U}\widehat{\mbPhi}^{*}({\mbLambda}){\bf U}^{\top}, \mbox{where}~
		\widehat{\mbPhi}^{*}({\mbLambda})=\mbox{diag}(\widehat{\phi}^{*}_{1, p}({\mbLambda}),\ldots, \widehat{\phi}^{*}_{p, p}({\mbLambda}))   \\  \nonumber
		&~\mbox{with}~ \widehat{\phi}^{*}_{i, p}({\mbLambda}) = \frac{\lambda_{i, p}}{1-\frac{p}{n}-2\frac{p}{n}
			\lambda_{i, p}\mbox{Re}[\check{m}^{\widehat{\mbtau}_{n}}_{n, p}({\lambda_{i, p}})]}, \forall i = 1, \ldots, p,
	\end{align}
	\vskip-0.3cm
	\noindent where $\check{m}^{\widehat{\mbtau}_{n}}_{n, p}(\cdot)$ is the estimator of $\check{m}_{F}(\cdot)$, 
	as well as a multivariate quantized eigenvalues sample function. Ledoit and Wolf [14] showed that
	$||\widehat{\mbSigma}^{-1}_{LW}-{\bf P}^{or}|| /\sqrt{p}\to 0$ a.s. with rescaled Frobenius norm 
	concluded that $\mbox{Re}[\check{m}^{\widehat{\mbtau}_{n}}_{n, p}({\lambda_{i, p}})]$ 
	is the consistent estimator of $\mbox{Re}[\check{ m}_{F}(\lambda_{i})], \forall~ i=1, \ldots, p$. 
	Ledoit and Wolf [16] concluded that 
	$0 < \widehat{\phi}^{*}_{1, p}({\mbLambda}) < \widehat{\phi}^{*}_{2, p}({\mbLambda})< \ldots	
	< \widehat{\phi}^{*}_{p, p}({\mbLambda})$ with probability one, further asserted the asymptotic optimality 
	of their shrinkage estimator (2.8) under Stein loss function.  \\
	\begin{comment}
	\indent Under assumptions A1-A4, \cite{[26]} proved that $F_{n}$ converges to $F$ a.s. as $n \to \infty$. Thus, $\lambda_{i, p}$ converges to $\lambda_{i}$ a.s., $i=1,\ldots, p$, and $\check{m}_{F_{n}} \to \check{m}_{F}$ a.s. as $n \to \infty$. \cite{[21]} used the sample quantile ($\lambda_{i, p}$) to estimate the corresponding population quantile (${\lambda}_{i}$) relating to the limiting distribution function $F$. Note that $F_{n} \to F$ a.s. if and only if $\check{m}_{F_{n}} \to \check{m}_{F}$ a.s. as $n \to \infty$, and hence $\widehat{\phi}_{i,p}({\mbLambda})$ in (\ref{eq:2.1}) converges to $1/a^{or}_{i}$ a.s. as $n \to \infty$, $i=1, \ldots, p$. 
	Hence we may note that both $\check{m}^{\widehat{\mbtau}_{n}}_{n, p} \to \check{m}_{F}$ a.s. and 
	$\check{m}_{F_{n}} \to \check{m}_{F}$ a.s., as $n \to \infty$. Namely, $\mbox{Re}[\check{m}^{\widehat{\mbtau}_{n}}_{n, p}({\lambda_{i, p}})] \to \mbox{Re}[\check{m}_{F}(\lambda_{i})]$ a.s. and 
	$\check{m}_{F_{n}}(\lambda_{i,p}) \to \mbox{Re}[\check{m}_{F}(\lambda_{i})]$ a.s., $i=1,\ldots,p$ as $n \to \infty$.\\
	\end{comment}
	\indent Ledoit and Wolf [16] pointed out that both the estimators in (2.1) and (2.8) 
	have a similar form in terms of Cauchy principal value. The only difference between (2.1) and 
	(2.8) is that the former uses the empirical distribution $F_{n}$ of sample
	eigenvalues to estimate the Stieltjes transform of distribution function $m_{F}(z)$, 
	while the latter one uses a smoothed version
	$F^{\widehat{\mbtau}_{n}}_{n, p}$ instead. They also commented that Stein's estimator in (2.1) has theoretical limitations and claimed that their estimator performs better compared to 
	Stein's estimator, by the evidence of Monte-Carlo simulations. 
	
	%%%%%%%%%% Section~3%%%%%%%%%%
	
	\section{Main results} %% p.7
	\subsection{The decomposite test statistic \texorpdfstring{$T^{2}_{N}$}{T2N}}
	%\subsection{The classes \texorpdfstring{$\mathcal{L}(\gamma)$}{Lg}}
	
	\indent For the problem (1.1), those tests proposed in the literature basically are made 
	by ignoring or using partial information from the sample covariance matrix. 
	The approach we adopt is to reveal the information of eigenvalues with the help of random matrix theory. 
	The orthogonally equivariant estimators of covariance matrix generally enjoy some
	optimal properties. The optimal one among the class of orthogonally equivariant estimators is mostly desired. 
	Ledoit and Wolf [14] claimed that $\widehat{\mbSigma}^{-1}_{LW}$ is 
	asymptotically optimal estimator of ${\mbSigma}^{-1}$ under Stein loss. 
	And hence, for the hypothesis testing problem (1.1) we may consider the test statistic
	\vskip-1.5cm
	\begin{align}~\label{eq:3.1} %% p.7
		T^{2}_{N}= n{\mXbar}^{\top}\widehat{\mbSigma}^{-1}_{LW}{\mXbar}. 
	\end{align}
	\vskip-5mm
	\hspace{-22pt}We may also note that $\widehat{\mbSigma}^{-1}_{LW}$ can be replaced by $\widehat{\mbSigma}^{-1}_{S}$, which is the
	inverse of matrix defined in (2.1).
	
	\indent Let ${\bf S}={\bf U}{\mbLambda}{\bf U}^{\top}=(s_{ij})$, take $\widehat{\mbPhi}^{*}(\mbLambda)$ 
	in (2.8) as (a) $\widehat{\mbPhi}_{0}(\mbLambda)={\bf I}$, 
	(b)~$\widehat{\mbPhi}_{1}(\mbLambda)=\mbox{diag}(s_{11},\ldots,s_{pp}), {\bf U}={\bf I}$, 
	(c)~$\widehat{\mbPhi}_{2}(\mbLambda)=\mbox{diag}(\lambda_{1,p},\ldots,\lambda_{p,p})$ and 
	(d)~$\widehat{\mbPhi}_{3}(\mbLambda)=\mbox{diag}(\lambda_{1,p}+\lambda,\ldots,\lambda_{p,p}+\lambda)$, 
	$\lambda >0$, then it~(3.1) reduces to the case of (a) Bai and Saranadasa [2], (b) Li et al. [17], 
	(c) the Hotelling's $\mbox{T}^{2}$-test~(1.3) and 
	(d) the regularized Hotelling's test Bai et al. [3] statistics, respectively.
	
	\indent  First, we may note that based on the results of Theorem~5 of Dempster [9], 
	${\bf u}^{\top}_{i}{\mbSigma}^{-1}{\bf u}_{i}=a^{*}_{i}$ are approximated by the quantity
	$a^{or}_{i}, i=1, \ldots, p$. Johnstone and Paul [11] proved that
	$\|\widehat{\mbSigma}^{-1}_{LW}-{\bf P}^{or}\|/\sqrt{p} \to 0$ a.s.
	under the rescaled Frobenius norm. Let ${\mbSigma}={\bf V}{\mbGamma}{\bf V}^{\top}$ and 
	${\mbSigma}_{1}={\bf V}{\mbGamma}^{*}{\bf V}^{\top}$, 
	where ${\mbGamma}=\mbox{diag}(\gamma_{1, p}, \ldots, \gamma_{p, p})$ 
	and ${\mbGamma}^{*}=\mbox{diag}(\gamma^{*}_{1, p}, \ldots, \gamma^{*}_{p, p})$ with
	$\gamma^{*}_{i, p}=1/a^{or}_{i},~i=1, \ldots, p$.
	
	\indent Since ${\bf A}=(n-1){\bf S}$ is Wishart distributed, when p is fixed we may note that ${\bf U}$ 
	is the maximum likelihood estimator (MLE) of ${\bf V}$ Muirhead [19]. 
	From the general theory of estimation that the maximum likelihood estimator is consistent, 
	it tends to the true value with probability one 
	as sample size becomes large under some regularity conditions which are satisfied by Wishart density. 
	When the dimension $p$ is fixed, we may conclude that 
	${\bf U}$ converges to ${\bf V}$ with probability one. Note that when $p$ is fixed and the sample size
	$n$ is large, $\widehat{\mbSigma}_{LW}$ reduces to the sample covariance matrix ${\bf S}$. Then
	$\widehat{\mbSigma}_{LW}$ converges to $\mbSigma$ with probability one. Namely, when dimension $p$ is fixed
	while the sample size $n$ is large, the decomposite $T^{2}_{N}$-test reduces to Hotelling's $T^{2}$-test statistic.
	Nevertheless, this optimal property remains wide open for large $p$ situation. 
	To overcome the difficulties, we also restrict
	the estimator of covariance matrix to the class of orthogonally equivariant estimators by imposing the sample
	eigenvector ${\bf u}_{i}$ on the corresponding population eigenvector ${\bf v}_{i}$ in this paper.
	
	\subsection{The asymptotically equivalent statistic of \texorpdfstring{$T^{2}_{N}$}{T2N} } %% p.9
	\indent The decomposite $T^{2}_{N}$-test statistic in (3.1) involves a mixture information of 
	nonlinear sample eigenvalues that complicates the task of deriving its distribution function. 
	By virtue of random matrix theory, Pan and Zhou [20] derived the limiting distribution function of 
	Hotelling's $T^{2}$-test statistic when ${\mbSigma}={\bf I}$. 
	Meanwhile, Chen et al. [5]  used Stieltjes transform to derive the asymptotic power distribution, 
	under $\mbox{H}_{0}$, of the regularized Hotelling's $T^{2}$-test statistic, which involves the 
	linear function of sample eigenvalues. Li et al. [17] extended the result for the one-sample 
	regularized Hotelling's $T^{2}$-test  of Chen et al. [5] to the two-sample problem 
	under a class of local alternatives. \\
	\indent Both the asymptotic power functions for the one-sample regularized Hotelling's $T^{2}$-test  of 
	Chen et al. (Theorem~1 and Proposition~2) and for the two-sample regularized Hotelling's $T^{2}$-test  of 
	Li et al. [17] are the functions of Stieltjes transform $\check{m}_{F}(x)$, defined in (2.4), 
	and its derivative. \\
	\indent Note that $\check{m}_{F}(x)$, when $x \in \mathcal{R}-\{0\}$, includes the real part $\mbox{Re}[\check{m}_{F}(x)]=\int_{-\infty}^{\infty}\dfrac{1}{t-x}\, dF(t)$ (Hilbert transform) and the imaginary part $\mbox{Im}[\check{m}_{F}(x)]=\pi f(x)$, where $f(x)=dF(x)/dx$.
	For example, when $\mbSigma={\bf I}$ the empirical density function 
	of eigenvalues converges weakly in probability to the Mar$\check{c}$enko-Pastur law
%	\vskip-0.8cm
%	\begin{equation*}
		$f(x)=\dfrac{1}{2\pi cx}\sqrt{(\xi_{+}-x)(x-\xi_{-})}\;,~x \in (\xi_{-} , \; \xi_{+})$,
%	\end{equation*}
%	\vskip-0.3cm
	where $\xi_{-}=(1-\sqrt{c})^{2}$ and $\xi_{+}=(1+\sqrt{c})^{2}$. 
	Then, with $-\lambda=x$ and $\gamma=c$ in equation (13) of Chen et al. [5] the Stieltjes transform of $F(x)$ becomes 
	\begin{align}\label{neq:3.2}  %% p.9
		\check{m}_{F}(x) = &{} \dfrac{(1-c-x)+\sqrt{(1-c-x)^{2}-4cx}}{2cx}\\
		= &{} \dfrac{(1-c-x)+\sqrt{-[(x-1-c)^2-4c]}}{2cx}\nonumber \\
		= &{} \dfrac{(1-c-x)+i\sqrt{(\xi_{+}-x)(x-\xi_{-})}}{2cx},~x \in (\xi_{-},\; \xi_{+}).\nonumber 
	\end{align}
	On the other hand, if we know the Stieltjes transform, we can from it deduce the limiting spectral density function 
%	\vskip-0.8cm
%	\begin{equation*}
		$f(x) = \dfrac{1}{\pi}\mbox{Im}[\check{m}_{F}(x)]$.\\
%	\end{equation*}
	\indent Thus both two asymptotic power functions of the regularized Hotelling's $T^{2}$-tests for one-sample problem 
	by Chen et al. [5] and for two-sample one by Li et al. [17]  are complex value functions, which seem to be against statistical 
	common sense for real test statistics.\\  %% p.10
	\indent Under assumptions A1-A4, Mar{\v c}enko and Pastur [18] proved that $F_{n}$ converges to $F$ a.s..
	It is well known that $m_{F_{n}}(z)$ converges to $m_{F}(z)$ a.s., $z \in \mathcal{C}^{+}$. Both Chen et al. [5] 
	and Li et al. [17] concluded this convergence holds even when $x \in \mathcal{R}-\left\{0\right\}$ and directly used $m_{F_{n}}(x)$ as the consistent estimator of $\check{m}_{F}(x)$ to prove Proposition~1 of Chen et al. [5], when $x \in \mathcal{R}-\left\{0\right\}$. However, when $x \in \mathcal{R}-\left\{0\right\}$ we may note that
	\vskip-1.3cm
	\begingroup
	\allowdisplaybreaks
	\begin{comment}
	\begin{align}~\label{eq:3.2}
	m_{F_n}(x) = &{} \dfrac{1}{p}\sum_{i=1}^{p}\dfrac{1}{\lambda_{i,p}-x}\\
	= &{} \int_{-\infty}^{\infty} \dfrac{1}{t-x}\.dF_{n}(t)\nonumber \\
	\to &{} \int_{-\infty}^{\infty} \dfrac{1}{t-x}\.dF(t)\nonumber \\
	= &{} \mbox{Re}[\check{m}_{F}(x)].\nonumber 
	\end{align}
	\end{comment}
	\begin{align}~\label{eq:3.2}  %%(3.3)
		m_{F_n}(x) =&{} \dfrac{1}{p}\sum_{i=1}^{p}\dfrac{1}{\lambda_{i,p}-x}
		= \int_{-\infty}^{\infty} \dfrac{1}{t-x}\,dF_{n}(t)\\[2ex]
		\to &{} \int_{-\infty}^{\infty} \dfrac{1}{t-x}\,dF(t)
		= \mbox{Re}[\check{m}_{F}(x)].\nonumber
	\end{align}
	\endgroup
	\vskip-0.5cm
	\noindent Hence $m_{F_n}(x)$ does not converge to $\check{m}_{F}(x)$ when $x \in \mathcal{R}-\left\{0\right\}$. 
	Thus, the Proposition~1 of Chen et al. [5] is not corrected and needs to be further re-investigated.\\  
	\indent To overcome the difficulties mentioned above we may instead try to find the 
	asymptotically equivalent statistic in distribution for $T^{2}_{N}$ called $T^{2}_{0}$, 
	defined in (3.6), which asymptotically local distribution and 
	asymptotically local power function can be acquired.\\
	\indent We assumed that the data come from the multinormal distribution, thus the sample mean vector 
	$\mXbar$ is independent of the sample covariance matrix ${\bf S}$, 
	namely $\mXbar$ is independent of ${\bf U}$ and $\mbLambda$.\\ 
	\indent Under assumptions A1-A4, Ledoit and Wolf [14] showed that 
	$\widehat\phi^{*}_{i,p}(\mbLambda)$ converges to 
	$1/a^{or}_{i}=\gamma^{*}_{i,p}$ a.s. as $n \to \infty$, $i=1,\ldots,p$.  %% p.11
	And they ([14], Proposition 4.3) further proved that $\|\widehat{\mbSigma}^{-1}_{LW}-
	{\bf U}\mbGamma^{*-1}{\bf U}^{\top}\|_{F} \to 0$ a.s. as $p \to \infty$ 
	( i.e., $\| \widehat{\mbPhi}^{*-1}(\mbLambda)-\mbGamma^{*-1}\|_{F} \to 0$ a.s. as $p \to \infty$), 
	where $\| \cdot \|_{F}$ is the Frobenius norm defined as $\| A \|_{F} = \sqrt{\mbox{tr}\left(AA^{\top}\right)/p}$ 
	with tr(A) denoting the trace of matrix A. \\ 
	\indent Since $\widehat{\phi}^{*}_{i,p}(\mbLambda)$converges to $\gamma^{*}_{i,p}$ 
	a.s. as $p \to \infty, i=1, \ldots,p$, namely, $\widehat{\phi}^{*-1}_{i,p}(\mbLambda) -\gamma^{*-1}_{i,p}= 
	o_{p}(1), i=1,\ldots,p,$ where $o_{p}(1) = \frac{1}{p^\gamma},~r \in (0,\infty)$. 
	Thus without loss of generality, for simplicity we may write 
	$\widehat{\phi}^{*-1}_{i,p}(\mbLambda) -\gamma^{*-1}_{i,p}= \frac{1}{p^\gamma}, i=1,\ldots,p$. 
	That is the same to say that 
	$\widehat{\mbPhi}^{*-1}(\mbLambda)-{\mbGamma^{*}}^{-1}=\frac{1}{p^\gamma}{\bf I}$, 
	as $p \to \infty$. Note that 
	$n{\mXbar}^{\top}{\bf U}\left\{{\widehat{\mbPhi^{*}}}^{-1}(\mbLambda)-{\mbGamma^{*}}^{-1}\right\}$
	${\bf U}^{\top}{\mXbar} = \frac{n}{p^\gamma}{\mXbar}^{\top}{\bf U}{\bf U}^{\top}{\mXbar} = 
	\frac{n}{p^\gamma}{\mXbar}^{\top}{\mXbar} \to \frac{1}{p^\gamma}\mbox{tr}(\mbSigma^{0})$ as $p \to \infty$, 
	where $\mbSigma^{0} = \mbSigma+n\mbmu\mbmu^{\top}$, $\mbSigma^{0}$ 
	is also  a positive definite matrix under local
	alternative (3.10). Decompose $\mbSigma^0$ as ${\bf V}^0\mbGamma^0{{\bf V}^{0}}^{\top}$, 
	where $\mbGamma^0 = \mbox{diag}(\gamma_{1,p}^{0}, \ldots, \gamma_{p,p}^{0})$ with
	$0<\gamma^0_{1,p}<\ldots<\gamma^0_{p,p}<\infty$. Let $b_{1}=\lim_{p \to \infty} \frac{1}{p^r}\mbox{tr}(\mbSigma^{0})$,
	we have the following three situations (i) $b_1 \to \infty$ when $\gamma \in (0,1)$, (ii) $b_1 \to 0$ when 
	$\gamma \in (1, \infty)$ and (iii) $b_1$ is a nonrandom but unknown constant when $\gamma = 1$. 
	Note that for the case (i), we may have $n\mXbar^{\top}\widehat{\mbSigma}^{*-1}\mXbar \to \infty$, 
	which is against statistical common sense. For the case (iii), $b_1 \to 0$ when $\gamma \in (0, \infty)$, 
	namely, it is the same as the fixed dimensional case. Hence without loss of generality, 
	we may only consider the case $\gamma = 1$ in details.
	\begin{comment}
	Decompose $\mbSigma^0$ as ${\bf V}^0\mbGamma^0{{\bf V}^{0}}^{\top}$, 
	where $\mbGamma^0 = \mbox{diag}(\gamma_{1,p}^{0}, \ldots, \gamma_{p,p}^{0})$ and $0<\gamma_{1,p}^{0} \gamma_{p,p}^{0}<\infty$.
	Let $b_{1}=\lim_{p \to \infty} \frac{1}{p^r}\mbox{tr}(\mbSigma^{0})$ is non-random but unknown. $\mbSigma^{0} = {\bf V}^{0}{\mbGamma}^{0}{\bf V^0}^{\top}, {\mbGamma}^{0} = \mbox{diag}(\gamma^{0}_{1,p},\ldots,\gamma^{0}_{p,p})$, where $0< \gamma^{0}_{1,p}< \ldots <\gamma^{0}_{p,p}< \infty$.
	\end{comment}
	Note that $0<p\gamma^{0}_{1,p}	\leq\mbox{tr}(\mbSigma^{0})\leq p\gamma^{0}_{p,p}$, 
	thus $0<\gamma^{0}_{1,p}\leq b_{1}\leq\gamma^{0}_{p,p}	< \infty$. 
	As such, $n{\mXbar}^{\top}{\bf U}\left\{{\widehat{\mbPhi^{*}}}^{-1}(\mbLambda)-{\mbGamma^{*}}^{-1}\right\}$
	${\bf U}^{\top}{\mXbar} \to b_{1}$ a.s. as $p \to \infty$.
	Hence for the high dimensional case, we may obtain that 
	$T^{2}_{N} = n{\mXbar}^{\top}{\bf U}{\widehat{\mbPhi^{*}}}^{-1}{\bf U}^{\top}{\mXbar} \to 
	n{\mXbar}^{\top}{\bf U}{{\mbGamma^{*}}}^{-1}{\bf U}^{\top}{\mXbar}+b_{1}$ in probability as $p \to \infty$.
	\begin{comment}
	For simplicity, we may take $\mbmu={\bf 0}$ and let $\widehat{\mbPhi}^{*-1}(\mbLambda)-\mbGamma^{*-1}=\dfrac{1}{p}{\bf I}_{p}.$
	Then we have that $n{\mXbar}^{\top}{\bf U}\left\{\widehat{\mbPhi}^{*-1}(\mbLambda)-\mbGamma^{*-1}\right\}{\bf U}^{\top}{\mXbar} = \dfrac{1}{p}n{\mXbar}^{\top}{\bf U}{\bf U}^{\top}{\mXbar}= \dfrac{1}{p}n{\mXbar}^{\top}{\mXbar}
	=\dfrac{1}{p}\mbox{tr}\;(\mbSigma)\;(1+\mbox{o}_{p}(1))$.
	Let $b_{1} = \lim_{p \to \infty}\dfrac{1}{p}\mbox{tr}\;(\mbSigma)$, which is nonrandom. 
	When $\mbSigma = {\bf I}$, then $b_{1}=1$. 
	Note that $0<p\gamma_{1,p}\leq\mbox{tr}\left(\mbSigma\right)=\sum_{i=}^{p}\gamma_{i,p}\leq p\gamma_{p,p}$, 
	thus $0<\gamma_{1,p}\leq b_{1}\leq \gamma_{p,p}<\infty$.
	As such, $n{\mXbar}^{\top}{\bf U}\left\{\widehat{\mbPhi}^{*-1}(\mbLambda)-\mbGamma^{*-1}\right\}{\bf U}^{\top}{\mXbar}\\
	\to b_{1}$ a.s. as $p \to \infty$. Hence we may obtain that 
	$T^{2}_{N} = n{\mXbar}^{\top}{\bf U}\widehat{\mbPhi}^{*-1}(\mbLambda){\bf U}^{\top}{\mXbar} \to n{\mXbar}^{\top}{\bf U}\mbGamma^{*-1}{\bf U}^{\top}{\mXbar} +b_{1}$ in probability as $p \to \infty$, where $0<\gamma_{1,p}\leq b_{1}\leq \gamma_{p,p}<\infty$. 
	\end{comment}
	This result implies that $T^{2}_{N} \to n{\mXbar}^{\top}{\bf U}\mbGamma^{*-1}{\bf U}^{\top}{\mXbar} +b_{1}$ 
	in distribution as $p \to \infty$. Namely,
	\vskip-1cm
	\begin{align*} 
		\dfrac{T^{2}_{N}-\mathcal{E}T^{2}_{N}}{\sqrt{{\mathcal Var}T^{2}_{N}}}\stackrel{\mathscr{D}}{\Longrightarrow}
		\dfrac{n{\mXbar}^{\top}{\bf U}\mbGamma^{*-1}{\bf U}^{\top}{\mXbar}+b_{1}
			-{\mathcal E}\left(n{\mXbar}^{\top}{\bf U}\mbGamma^{*-1}{\bf U}^{\top}{\mXbar} +b_{1}\right)}
		{\sqrt{{\mathcal Var}\left(n{\mXbar}^{\top}{\bf U}\mbGamma^{*-1}{\bf U}^{\top}{\mXbar} +b_{1}\right)}} ~~\mbox{as} ~p \to \infty.
	\end{align*}
	\indent Furthermore, it is easy to note that  %%p.12
	\begin{comment}
	\begin{align*}
	&\dfrac{n{\mXbar}^{\top}{\bf U}\mbGamma^{*-1}{\bf U}^{\top}{\mXbar}+b_{1}
	-{\mathcal E}\left(n{\mXbar}^{\top}{\bf U}\mbGamma^{*-1}{\bf U}^{\top}{\mXbar} +b_{1}\right)}
	{\sqrt{{\mathcal Var}\left(n{\mXbar}^{\top}{\bf U}\mbGamma^{*-1}{\bf U}^{\top}{\mXbar} +b_{1}\right)}} \\
	= &{} \dfrac{n{\mXbar}^{\top}{\bf U}\mbGamma^{*-1}{\bf U}^{\top}{\mXbar}
	-{\mathcal E}\left(n{\mXbar}^{\top}{\bf U}\mbGamma^{*-1}{\bf U}^{\top}{\mXbar}\right)}
	{\sqrt{{\mathcal Var}\left(n{\mXbar}^{\top}{\bf U}\mbGamma^{*-1}{\bf U}^{\top}{\mXbar}\right)}}.
	\end{align*}
	\end{comment}
	under normalization both two statistics $n{\mXbar}^{\top}{\bf U}\mbGamma^{*-1}{\bf U}^{\top}{\mXbar}+b_{1}$ 
	and $n{\mXbar}^{\top}{\bf U}\mbGamma^{*-1}{\bf U}^{\top}{\mXbar}$ have the same asymptotic distribution function as $p \to \infty$.
	Thus we may obtain that  
	\begin{align*} 
		\dfrac{T^{2}_{N}-\mathcal{E}T^{2}_{N}}{\sqrt{{\mathcal Var}T^{2}_{N}}}\stackrel{\mathscr{D}}{\Longrightarrow}
		\dfrac{n{\mXbar}^{\top}{\bf U}{\mbGamma}^{*-1}{\bf U}^{\top}{\mXbar}-
			\mathcal{E}\left(n{\mXbar}^{\top}{\bf U}{\mbGamma}^{*-1}{\bf U}^{\top}{\mXbar}\right)}
		{\sqrt{{\mathcal Var}\left(n{\mXbar}^{\top}{\bf U}{\mbGamma}^{*-1}{\bf U}^{\top}{\mXbar}\right)}}~~\mbox{as} ~p \to \infty.
	\end{align*}
	Hence we may have the following conclusion:
	\vskip-1.6cm
	\begin{comment}
	\begin{align}~\label{eq:3.3}
	T^{2}_{N}=&{} n{\mXbar}^{\top}{\bf U}\widehat{\mbPhi}^{*-1}(\mbLambda){\bf U}^{\top}{\mXbar}  \\\nonumber
	\to &{} n{\mXbar}^{\top}{\bf U}\mbGamma^{*-1}{\bf U}^{\top}{\mXbar}
	\end{align}
	\end{comment}
	\begin{align}~\label{eq:3.3}
		T^{2}_{N}= ~n{\mXbar}^{\top}{\bf U}\widehat{\mbPhi}^{*-1}(\mbLambda){\bf U}^{\top}{\mXbar}  
		\to~n{\mXbar}^{\top}{\bf U}\mbGamma^{*-1}{\bf U}^{\top}{\mXbar}
	\end{align}
	\vskip-0.8cm
	\hspace{-22pt}in distribution as $p \to \infty$.
	
	\indent When $\mbSigma={\bf I}$, from~(3.2) we may note that $\mbox{Re}[\check{m}_{F}(x)]=(1-c-x)/(2cx)$ 
	and then from~(2.7) we have that $1/a^{or}_{i}=\gamma^{*}_{i, p}=1,~i = 1, \ldots, p$, 
	i.e., $\Gamma^{*}={\bf I}$. Thus, by equation~(3.4) we may obtain that 
	$T^{2}_{N} \to n{\mXbar}^{\top}{\bf U}{\bf U}^{\top}{\mXbar}=n{\mXbar}^{\top}{\mXbar}\sim \chi_{p}^{2}(n{\mbmu}^{\top}{\mbmu})$ 
	in distribution as $p \to \infty$, where $\chi_{p}^{2}(n{\mbmu}^{\top}{\mbmu})$ 
	is non-central chi-square distributed with 
	$p$ degrees of freedom and non-centrality $n{\mbmu}^{\top}{\mbmu}$. 
	Therefore, we have the following theorem.
	%\vspace{0.3cm}
	
	\begin{theorem}%Theorem 1  (p.12)
		Under assumptions A1-A4, when $\mbSigma={\bf I}$, then $T^{2}_{N}$ 
		is asymptotically equivalent to $\chi_{p}^{2}(n{\mbmu}^{\top}{\mbmu})$ in distribution as $p \to \infty$.
	\end{theorem}	
	For the hypothesis testing problem~(1.1), when $\mbSigma={\bf I}$ 
	Theorem~1 indicates that the decomposite $T^{2}_{N}$-test may be 
	asymptotically optimal when the dimension is large, while the Hotelling's $T^2$-test is not.\\
	\indent However, the situation may be different for general $\mbSigma$. If ${\bf U}$ as the consistent estimator of ${\bf V}$ can be true, we then have $\widehat{\mbSigma}^{-1}_{LW}$ 
	converges to ${\mbSigma}^{-1}_{1}$ with probability one. 
	Despite that we adopt the Ledoit and Wolf's optimal estimator (eq:2.8) 
	of the population precision matrix, however, ${\mbSigma}_{1}$ may not generally be equal to ${\mbSigma}$, 
	unless that ${\mbGamma}^{*}={\mbGamma}$ (i.e., ${1}/{{a^{or}_{i}}}=\gamma_{i,p}, \forall i=1, \ldots, p$). 
	The orthogonal matrix $\bf U$ may not generally be a consistent estimator of $\bf V$ when the dimension $p$ 
	is large, (see Bai et al. [4] and references therein). Hence we may work it under the restricted model, 
	namely, under the Wishart distribution setup when $p/n \to c \in (0, 1)$. \\	
	\indent Note that ${\bf U} \in \mathcal{Q}(p)$, the  group of $p \times p$ orthogonal matrices, 
	which is a compact group. 
	Hence, there exists a subsequence $\{n_{1}\}$ such that ${\bf U}_{n_{1}}$ 
	converges to ${\bf V}$ a.s. as $n \to \infty$. 
	For the case $c\in (0,1)$, when $\mbSigma = {\bf I}$ we obtain that $\mbGamma^{*}={\bf I}$ as $p\to \infty$, 
	and hence we may have that 
	$n{\mXbar}^{\top}{\bf U}\mbGamma^{*-1}{\bf U}^{\top}{\mXbar} = n{\mXbar}^{\top}{\bf U}{\bf U}^{\top}{\mXbar}=
	n{\mXbar}^{\top}{\mXbar}=n{\mXbar}^{\top}{\bf V}{\bf V}^{\top}{\mXbar}=n{\mXbar}^{\top}{\bf V}\mbGamma^{*-1}{\bf V}^{\top}{\mXbar} $. Thus by the equation~(3.4), when $\mbSigma = {\bf I}$ 
	we have that $T^{2}_{N}\to n{\mXbar}^{\top}{\bf V}{\mbGamma}^{*-1}{\bf V}^{\top}{\mXbar}$ 
	in distribution as $p \to \infty$.\\
	\indent In the general $\mbSigma$ case, for simplicity we may investigate the limiting distribution function of 
	$T^{2}_{N}$ under the assumption that ${\bf U}$ converges to  ${\bf V}_{1}$ a.s. in the weak topology as 
	$n \to  \infty$, ${\bf V}_{1} \in \mathcal{Q}(p)$. 
	By equation~(3.4) we then have
	\vskip-1.3cm 
	\begin{align}~\label{eq:3.4} %%p.13
		T^{2}_{N} \to n{\mXbar}^{\top}{\bf V}_{1}{\mbGamma}^{*-1}{\bf V}_{1}^{\top}{\mXbar}
	\end{align}
	\vskip-0.5cm
	\noindent in distribution as $p \to \infty$. \\
	\indent Furthermore, since $\mathcal{Q}(p)$ is one orbit, and hence by the theory of compact group there exist some 
	${\bf G} \in \mathcal{Q}(p)$ with ${\bf G}^{\top}{\bf V}_{1}={\bf V}$. Therefore
	\vskip-1.5cm
	\begin{align}~\label{eq:3.5}
		T^{2}_{N}\to  n{\mXbar}^{\top}{\bf G}{\bf V}{\mbGamma}^{*-1}{\bf V}^{\top}{\bf G}^{\top}{\mXbar}
		= n{\mXbar}^{\top}{\bf G}{\mbSigma}^{-1}_{1}{\bf G}^{\top}{\mXbar}
		= T^{2}_{0}, ~\mbox{say}
	\end{align}
	\vskip-0.5cm
	\noindent in distribution as $p \to \infty$. Note that 
	${\bf G}{\mbSigma}^{-1}_{1}{\bf G}^{\top} \neq \mbSigma^{-1}$ generally. 
	This will make things for high-dimensional situations different from those for the fixed dimensions. 
	Both Stein [26] and Ledoit and Wolf [14] directly considered the case when ${\bf G} = {\bf I}$. 
	Namely, when $\bf U$ converges 
	to $\bf V$ a.s. as $p \to \infty$.
	\begin{theorem} %%% Theorem 2 (p.14)
		Assume that ${\bf U}$ converges to ${\bf V}_{1}$ a.s. in the weak topology as 
		$n \to \infty$, ${\bf V}_{1} \in \mathcal{Q}(p)$. 
		Then under assumptions A1-A4, $T^{2}_{N}$ is asymptotically equivalent to $T^{2}_{0}$ (
		defined in (3.6)) in distribution as $p \to \infty$.
	\end{theorem}
	
	%%%%%%%%%
	% Section #3.3  p.14
	%%%%%%%%%
	
	\subsection{The asymptotic distribution of \texorpdfstring{$T^{2}_{N}$}{T2N} } %% p.14
	\indent Let ${\bf Z}=\sqrt{n}{\mbSigma}^{-1/2}{\mXbar}$, then ${\bf Z} \sim N(\sqrt{n}{\mbSigma}^{-1/2}{\mbmu},{\bf I})$.
	Let ${\mbSigma}_{2}=\mbSigma^{1/2}{\bf G}\mbSigma^{-1}_{1}{\bf G}^{\top}\mbSigma^{1/2}$, and decompose it as 
	${\bf V}_{2}\mbPsi (\mbGamma){\bf V}_{2}^{\top}$, 
	where $\mbPsi (\mbGamma)=\mbox{diag}(\psi_{1},\ldots,\psi_{p}),~0<\psi_{1}\le\psi_{2}\le\ldots\le\psi_{p}<\infty$. 
	Let ${\bf W}={\bf V}_{2}{\bf Z}$, then ${\bf W} \sim N(\mbox{\btheta},{\bf I})$, where \mbtheta=$\sqrt{n}{\bf V}_{2}{\mbSigma}^{-1/2}\mbmu=(\theta_1,\ldots,\theta_p)^{\top}$. 
	And
	\vskip-1.2cm
	\begin{align}\label{eq:3.6} %%(3.7)
		T^{2}_{0}= {\bf Z}^{\top}{\mbSigma}_{2}{\bf Z}
		={\bf W}^{\top}{\mbPsi}(\mbGamma){\bf W}
		=\sum_{i=1}^{p}{\psi}_{i}w_{i}^{2},	
	\end{align}
	\vskip-0.5cm
	\noindent which is the mixture of non-central chi-square distributions. By the results of 
	Corollary 1.3.5 of Muirhead [19], after some straightforward algebraic calculations, we have
	\vskip-1.2cm
	\begin{align}
		{\mathcal E} T^{2}_{0}&=\sum_{i=1}^{p}\psi_{i}+\sum_{i=1}^{p}\psi_i{\theta}^{2}_{i}~\label{eq:3.7} %%(3.8)
	\end{align}
	\vskip-0.5cm
	\noindent and
	\vskip-1.5cm
	\begin{align}
		{\mathcal Var} T^{2}_{0}&=2\sum_{i=1}^{p}\psi_{i}^{2}+4\sum_{i=1}^{p}\psi^{2}_{i}{\theta}^{2}_{i}. %%(3.9)
	\end{align}
	\vskip-0cm %%(p.14)
	\indent Generally, the power of any reasonable test goes to one when the sample size $n$ 
	is large (Chen et al. [5],Theorem~1).
	Thus, it is hard to compare the tests when the sample size $n$ goes to infinity. As such, we may use the local power to compare the tests. We extent the concept of local power from the fixed dimensional situation to the large dimensional one.
	Different from the fixed dimensional one, we incorporate the dimension $p$ into the consideration for the large dimensional situation. We study the asymptotic distribution of $T^{2}_{N}$ under the sequence of local alternatives
	\vskip-1.5cm
	\begin{align}~\label{eq:3.8}
		H_{0}:{\mbmu}={\bf 0}~\mbox{versus}~H_{1n}:{\mbmu}=n^{-1/2}p^{1/4}{\mbdelta},
	\end{align}
	\vskip-0.5cm
	\noindent where $\mbdelta$ is a fixed $p$-dimensional vector, which means to assume that 
	${\mbdelta}^{\top}{\mbSigma}^{-1}{\mbdelta} < \infty$ when $p$ is large. 
	We may remark that this local alternative is equivalent to the one $\|{\mbmu}\|=O(n^{-1/2}p^{1/4})$ 
	considered in Feng et al. [10].\\
	\indent Let $\mbmu=n^{-1/2}p^{1/4}\mbdelta$, then \mbtheta = $p^{1/4}$\mbbeta, 
	where \mbbeta=${\bf V}_{2}\mbSigma^{-1/2} {\mbdelta}$. 
	Note that $\sum_{i=1}^{p}\psi_{i}\,\beta^{2}_{i}\le\psi_{p}\sum_{i=1}^{p}\beta^{2}_{i}=\psi_{p}{\mbdelta}^{\top}
	\mbSigma^{-1}{\mbdelta}<\infty$ and $\sum_{i=1}^{p}\psi^{2}_{i}\beta^{2}_{i}\le
	\psi^{2}_{p}\sum_{i=1}^{p}\beta^{2}_{i}=\psi^{2}_{p}{\mbdelta}^{\top}\mbSigma^{-1}{\mbdelta}<\infty$.
	Thus, $p^{-1/2}\sum_{i=1}^{p}\psi_{i}\,\beta^{2}_{i} \to 0$ and $p^{-1/2}\sum_{i=1}^{p}\psi^{2}_{i}\beta^{2}_{i} \to 0$ 
	as $p \to \infty$.	
	\begin{theorem} %Theorem~3 (p.15)
		Under the assumptions of Theorem~2 and the sequence of local alternatives $H_{1n}$ defined in (3.10), 
		the asymptotic power function of test statistic $T^{2}_{N}$ in (3.1) is
		\vskip-1cm
		\begin{align}~\label{eq:3.9}
			\beta(\mbmu) \approx \Phi\left(-z_{{\alpha }}+(2d)^{-1/2}\sum_{i=1}^{p}\psi_{i}\,\beta^{2}_{i}\right),
		\end{align}
		\hspace{-4pt} where $\Phi(\cdot)$ denotes the standard normal distribution, 
		and $d= \lim_{p \to \infty} \sum_{i=1}^{p}\psi^{2}_{i}/p$ being a positive constant. 
	\end{theorem}
	\begin{proof}
		Under the null hypothesis, we may note that ${\mathcal E} T^{2}_{0}=\sum_{i=1}^{p}\psi_{i}$ and ${\mathcal Var}
		T^{2}_{0}=2\sum_{i=1}^{p}\psi^{2}_{i}$. Thus by Theorem~2 as $n \to \infty$ we have
		\begingroup
		\allowdisplaybreaks
		\begin{align}
			&P\left\{\dfrac{T^{2}_{N}-{\mathcal E} T^{2}_{N}}{\sqrt{{\mathcal Var} T^{2}_{N}}}\geq z_{{\alpha }}\Biggr|H_{0}\right\} \\[1.5pt] \nonumber
			\approx {}& P\left\{\dfrac{T^{2}_{0}-{\mathcal E} T^{2}_{0}}{\sqrt{{\mathcal Var} T^{2}_{0}}}\geq z_{{\alpha }}\Biggr|H_{0}\right\}\\[1.5pt]\nonumber
			= {}& P\left\{\dfrac{T^{2}_{0}-\sum_{i=1}^{p}\psi_{i}}{\sqrt {2\sum_{i=1}^{p}\psi^{2}_{i}}}\geq  z_{\alpha}\right\} \\\quad\nonumber
			\approx {}& 1-\Phi(z_{{\alpha}}) \\[1.5pt] \nonumber
			= {}& \Phi(-z_{{\alpha}})\\[1.5pt] \nonumber
			= {}& \alpha. \nonumber 
		\end{align}
		\endgroup
		%\newpage
		\lhead[\footnotesize\thepage\fancyplain{}\leftmark]{}\rhead[]{\fancyplain{}\rightmark\footnotesize\thepage}%Put this line in Page 2
		\noindent And hence, under the sequence of local alternatives $H_{1n}$ we then have
		\begingroup
		\allowdisplaybreaks
		\begin{align}~\label{eq:3.10}
			&P\left\{\dfrac{T^{2}_{N}-{\mathcal E} T^{2}_{N}}{\sqrt{{\mathcal Var} T^{2}_{N}}}\geq z_{{\alpha }}\Biggr|H_{1n}\right\}\\[2pt] \nonumber
			\approx {}& P\left\{ \dfrac{T^{2}_{0}-\sum_{i=1}^{p}\psi_{i}}{\sqrt {2\sum_{i=1}^{p}\psi^{2}_{i}}} \geq z_{{\alpha }}\Biggr|H_{1n} \right\} \\[2pt] \nonumber
			={}& P\left\{ \dfrac{T^{2}_{0}-(\sum_{i=1}^{p}\psi_{i}+p^{1/2}\sum_{i=1}^{p}\psi_{i}\,\beta^{2}_{i})}
			{\sqrt {2\sum_{i=1}^{p}\psi^{2}_{i}}} \geq z_{{\alpha }}-\frac{p^{1/2}\sum_{i=1}^{p}\psi_{i}\,\beta^{2}_{i}}
			{\sqrt {2\sum_{i=1}^{p}\psi^{2}_{i}}}\Biggr|H_{1n}\right\} \\[2pt] \nonumber
			\approx {}&  \Phi \, \Biggr(-z_{\alpha}{\sqrt {\left(1+2\dfrac{p^{1/2}\sum_{i=1}^{p}\psi^{2}_{i}\beta^{2}_{i}}{\sum_{i=1}^{p}\psi^{2}_{i}}\right)^{-1}}}+
			\frac{p^{1/2}\sum_{i=1}^{p}\psi_{i}\,\beta^{2}_{i}}
			{\sqrt {2{\sum_{i=1}^{p}\psi^{2}_{i}+4p^{1/2}\sum_{i=1}^{p}\psi^{2}_{i}\beta^{2}_{i}}}}\Biggr) \\[2pt]\nonumber
			\to {}&  \Phi  \, \Biggr(-z_{\alpha}+\frac{p^{1/2}\sum_{i=1}^{p}\psi_{i}\,\beta^{2}_{i}}
			{\sqrt {2{\sum_{i=1}^{p}\psi^{2}_{i}+4p^{1/2}\sum_{i=1}^{p}\psi^{2}_{i}\beta^{2}_{i}}}}\Biggr)\\[2pt] \nonumber
			={}& \Phi \, \Biggr(-z_{\alpha}+\dfrac{\sum_{i=1}^{p}\psi_{i}\,\beta^{2}_{i}}
			{\sqrt{2p^{-1}\sum_{i=1}^{p}\psi^{2}_{i}+4p^{-1/2}\sum_{i=1}^{p}\psi^{2}_{i}\beta^{2}_{i}}}\Biggr)\\[2pt] \nonumber
			\to{}& \Phi \, \Biggr(-z_{\alpha}+\dfrac{\sum_{i=1}^{p}\psi_{i}\,\beta^{2}_{i}}{\sqrt{2d}}\Biggr).
		\end{align}
		\endgroup
	\end{proof}	
	\par
	Generally, the applications of Theorem~3, it needs the consistent estimators of $\psi_{i},~i=1,\ldots,p$. When $\mbSigma={\bf I}$, then $\psi_{i}=a^{or}_{i}$, which can be consistently estimated by 
	$\widehat{\phi^{*}_{i,p}}^{-1},i=1,\ldots,p$. Hence we have the following.	
	\begin{corollary}%corollary~1 (p.16)
		Under the assumptions of Theorem~2 and under $\mbox{H}_{0}$, when $\mbSigma={\bf I}$, we have that
		$\dfrac{T^{2}_{N}-\sum_{i=1}^{p}\widehat{\psi}_{i}}{\sqrt{2\sum_{i=1}^{p}\widehat{\psi_{i}}^{2}}}=
		\dfrac{T^{2}_{N}-\sum_{i=1}^{p}\widehat{\phi^{*}_{i,p}}^{-1}}{\sqrt{2\sum_{i=1}^{p}\widehat{\phi^{*}_{i,p}}^{-2}}}
		\longrightarrow \mbox{N}(0,1).$
		\begin{comment}
		\vskip-1.5cm
		\begin{align*}
		\dfrac{T^{2}_{N}-\sum_{i=1}^{p}\widehat{\psi}_{i}}{\sqrt{2\sum_{i=1}^{p}\widehat{\psi_{i}}^{2}}}=
		\dfrac{T^{2}_{N}-\sum_{i=1}^{p}\widehat{\phi^{*}_{i,p}}^{-1}}{\sqrt{2\sum_{i=1}^{p}\widehat{\phi^{*}_{i,p}}^{-2}}}
		\longrightarrow \mbox{N}(0,1).
		\end{align*}
		\end{comment}
	\end{corollary}
	\noindent Thus, when $\mbSigma={\bf I}$ the quantity $\left(T^{2}_{N}-\sum_{i=1}^{p}\widehat{\psi}_{i}\right)\bigg/\sqrt{2\sum_{i=1}^{p}\widehat{\psi_{i}}^{2}}$ is completely data-driven.\\
	\indent Let ${\bf D}={\mbGamma}{\mbGamma}^{*-1}$, write ${\bf D}$=diag$(d_{1},\ldots,d_{p})$. 
	The weight $d_i$ is the ratio of $i$th eigenvalues of two covariance matrices 
	${\mbSigma}$ and ${\mbSigma}_{1}$, i.e., 
	$d_{i}=\gamma_{i, p}/ \gamma^{*}_{i, p}=\gamma_{i,p} a^{or}_{i},~i=1,\dots, p$. 
	Then we may note that $0 < a^{or}_{i} < \infty$ 
	and hence $0 < d_{i} < \infty,\forall i=1,\ldots,p$.\\
	\indent When ${\bf G}={\bf I}$, (i.e.,${\bf V}_{1}={\bf V}$), 
	then $\mbSigma_{2}$=${\bf V}{\bf D}{\bf V}^{\top}$, $\mbPsi(\mbGamma)={\bf D}$ 
	and ${\bf V}_{2}={\bf V}$. Hence we have the following.
	\begin{corollary}%corollary~2  (p.17)
		For the hypothesis testing problem (1.1), under the assumptions of Theorem~2, 
		if ${\bf G}={\bf I}$ (i.e., ${\bf U} ~\mbox{converges to } {\bf V}$ a.s., as $n \to \infty$), then the asymptotically 
		local power of ~$T^{2}_{N}$ is $	\beta(\mbmu) \approx \Phi \left(-z_{{\alpha }}+(2d)^{-1/2}
		\sum_{i=1}^{p}d_{i}\,\beta^{2}_{i}\right) = \Phi\left(-z_{{\alpha }}+(2d)^{-1/2}{\mbdelta}^{\top}
		{\mbSigma^{-1}_{1}}{\mbdelta}\right).$
		\begin{comment}
		\begin{align*}
		\beta(\mbmu) \approx &{} \Phi \left(-z_{{\alpha }}+(2d)^{-1/2}\sum_{i=1}^{p}d_{i}\,\beta^{2}_{i}\right)
		= \Phi\left(-z_{{\alpha }}+(2d)^{-1/2}{\mbdelta}^{\top}{\mbSigma^{-1}_{1}}{\mbdelta}\right).
		\end{align*}
		\end{comment}
	\end{corollary}
	\indent The statistic $T^{2}_{N}$ asymptotically reduces to non-central chi-square distributed when $d_{i}=1,~\forall i=1,\ldots,p$,
	(i.e., $\mbSigma_{1}=\mbSigma$).
	\begin{corollary} %corollary3  (p.17)
		For the hypothesis testing problem (1.1), under the assumptions of Theorem~2,
		if ${\bf D}={\bf I}$, then the proposed $T^{2}_{N}$-test is asymptotically optimal.
	\end{corollary}	
%	\begin{remark} % Remark 1 (p.17)
		{\bf Remark~1}
		For one thing, if $d=\infty$, then the asymptotic power of $T^{2}_{0}$ is equal to the significant level $\alpha$. 
		And for another, as ${\mbGamma}={\mbGamma}^{*}$ the diagonal matrix ${\bf D}$ 
		equals to ${\bf I}$, i.e., ${\mbSigma}_{1}={\mbSigma}$.
		If ${\bf D}={\bf I}$ (i.e., $d=1$) the proposed test $T^{2}_{N}$ has the asymptotically optimal power property. 
		Moreover, we may note that the asymptotic distribution of the optimal test statistic 
		is non-central $\chi^{2}$ distributed.
		As a result, the key point to obtain the asymptotically optimal Hotelling's type test 
		is to use the consistent estimator of ${\mbSigma}$.
%	\end{remark}
	\indent Note that we assume that $n > p$, then $\bf S$ is the MLE of $\mbSigma$. 
	And hence ${\bf S} \to \mbSigma$ a.s. when $p$ is fixed and $n \to \infty$ (i.e., $c=0$). 
	Thus Hotelling's $T^{2}$-test statistic in (1.3) converges to 
	$n\overline{\bf X}^{\top}\mbSigma^{-1}\overline{\bf X}$ in probability as $n \to \infty$. 
	However, it may not be true when $c \in (0,1)$ due to the inconsistency of sample eigenvalues. 
	By Theorem 3 and Remark 1, we have the following
	\begin{corollary} %corollary4  (p.18)
		For the hypothesis testing problem (1.1), the Hotelling's $T^{2}$-test is 
		asymptotically optimal when $c=0$. However, it is not asymptotically optimal when $c \in (0,1)$. 
	\end{corollary}
%	\begin{remark}% Remark 2  (p.18)
		{\bf Remark~2} Generally $\mbSigma_{1}\neq {\mbSigma}$, by Corollary~3 the 
		decomposite test statistic $T^{2}_{N}$, 
		which is based on the optimal orthogonally equivariant estimator 
		$\widehat {\mbSigma}^{-1}_{LW}$ for the precision matrix
		${\mbSigma}^{-1}$, will not be asymptotically optimal for the hypothesis testing problem 
		(1.1). Namely, all the regularized Hotelling's $T^{2}$ type tests are not asymptotically optimal 
		due to the sample eigenvalues inconsistency. As such, to obtain the asymptotically optimal test for  
		the hypothesis testing problem (1.1) without having the structure assumption of covariance, 
		it is necessary to do more modification work with the eigenvalue and eigenvector estimation of 
		population covariance matrix ${\mbSigma}$. Namely, to find out the consistent estimator of $\mbSigma$ 
		when $c \in (0,1)$ is a quite hard work. It remains wide open in the literature. 
%	\end{remark}
	
%	\begin{remark}%Remark 3   (p.18)
		{\bf Remark~3} Usually, for the fixed dimensional cases there is no any restriction 
		on the unknown nuisance parameter $\mbSigma$ to 
		establish the asymptotic normality of the test statistics.
		However, for the large dimensional p cases, the asymptotic normality of the test statistics holds 
		either under some restrictions on the unknown nuisance parameter $\mbSigma$ 
		or the case that proposed test is optimal when 
		${\bf D}={\bf I}$. As such, the numerical powers of tests under 
		\begin{comment}
		We extend the local alternatives from finite dimension case to the large dimensional
		one (\ref{eq:3.8}) by incorporating the dimensionality $p$, it is still not sufficient to guarantee the asymptotic normality
		of test statistics. For large p cases, the asymptotic normality of test statistics holds either under some restrictions on the
		unknown nuisance parameter ${\mbSigma}$ or the case that proposed test is optimal when ${\bf D}={\bf I}$. Usually,
		for the fixed dimension cases there is no restriction on the nuisance parameter $\mbSigma$, but the numerical powers
		of tests under 
		\end{comment}
		large dimension situation are not comparable. Because we can only perform those numerical power functions
		under restricted parameter spaces of ${\mbSigma}$, where the asymptotic normality of test statistics holds.
		Those restriction spaces of ${\mbSigma}$ over spaces
		$\{\mbSigma| \mbSigma > {\bf 0}, \mbox{tr}({\mbSigma}^{4})/\mbox{tr}^{2}({\mbSigma}^{2})=o(1)\}$ and
		$\{\mbSigma| \mbSigma > {\bf 0}, \mbox{tr}({\mbGamma}^{4}_{K})=o(\mbox{tr}^{2}({\mbGamma}^{2}_{K}))\}$
		Feng et al. [10], are generally hard to be analytically characterized. Each test may have different restricted
		parameter space to ensure the asymptotic normality of test statistic. Besides, there is no clear way to compare the 
		power functions for those tests beyond restricted spaces.
		To overcome the difficulty, we provide a testing procedure under the local alternative which 
		the dimensionality p is also taken into the consideration. This generalize the fixed dimensional 
		situations into the large dimensional cases. 
		Our proposed test statistics $T^{2}_{N}$ dose not encounter such a disaster mentioned above, as we have discussed in Corollary~3, the optimal convergence estimator of ${\mbSigma}$ will lead
		the corresponding test to be optimal. Thus, to compare the tests for hypothesis testing problem (1.1),
		 it is essential to compare the estimators of ${\mbSigma}$. Random matrix theory will play 
		 an important role in obtaining reasonable
		estimators of population covariance matrix. We will explain this point more clearly through 
		comparisons with the existing tests in Section 4.
%	\end{remark}
	
	%%%%%%%%%%%%%%%%%%%%%%%%%%%%
	\section{The comparison of tests} %% p.19
	\subsection{The asymptotic relative efficiency}
	\indent A standard method to compare asymptotic power functions is through asymptotic relative efficiency (ARE) 
	(Pitman [22]), which is essentially defined via large deviation asymptotics. 
	It is well known that the Sanov theorem and its generalizations reduce the problem of large deviations 
	to a minimization problem of Kullback-Leibler divergence on the corresponding set of distributions. 
	For any two test statistics which are asymptotic to normal, i.e., $\chi^{2}$
	distributed with noncentralities $\mu^{'}{\bf A}{\mu}$ and $\mu^{'}{\bf B}{\mu}$, respectively. Then the ARE of
	these two tests is equivalent to $\mu^{'}{\bf A}{\mu}/\mu^{'}{\bf B}{\mu}$. Whenever the value of ARE of test $T_{a}$
	relative to test $T_{b}$ is larger than one, then the procedure based on $T_{a}$ is considered to have larger
	asymptotic power than that of the competing test based on $T_{b}$. The test $T_{a}$ has the better 
	asymptotic power than that of test $T_{b}$ if the eigenmatrix of ${\bf A}{\bf B}^{-1}$ is larger than ${\bf I}$. Following the arguments
	as in Case 1, we can easily see that the tests proposed by
	 Dempster [8], [9], Bai and Saranadasa [2], Srivastava and Du [23], 
     Srivastava [24], Chen and Qin [6], Chen et al. [5], Park and Ayyala [21]  and Feng et al. [10] 
     are not optimal for the hypothesis testing problem (1.1) when the dimension is large. 
     Basically, these results can be classified into the following three categories:
		
	\vspace{0.3cm} %%p.20
	\indent {\bf Case 1.} Compare the ARE of tests constructed without using the information of correlations. 
	Let ${\bf B}^{-1}{\bf A}= \left[\mbox{tr}(\mbSigma^{2})\right]^{1/2}\textbf{I}\mbSigma^{-1} =
	\sqrt{\mbox{tr}(\mbSigma^{2})}\mbSigma^{-1}.$
	\begin{comment}
	\vskip-1.5cm
	\begin{align} \nonumber
	{\bf B}^{-1}{\bf A}= \left[\mbox{tr}(\mbSigma^{2})\right]^{1/2}\textbf{I}\mbSigma^{-1}
	= \sqrt{\mbox{tr}(\mbSigma^{2})}\mbSigma^{-1}.
	\end{align}
	\vskip-0.5cm
	\end{comment}
	Thus the eigenmatrix of ${\bf B}^{-1}{\bf A}$ is larger than ${\bf I}$.
	\begin{comment}
	\begin{equation}~\label{eq:4.1}
	\left(\sum_{i=1}^{p}\gamma^{2}_{i, p}\right)^{1/2}
	\begin{bmatrix}
	\dfrac{1}{\gamma_{1, p}} & \cdots & 0\\
	\vdots                 & \ddots & \vdots\\
	0                      & \cdots & \dfrac{1}{\gamma_{p, p}}
	\end{bmatrix}
	=\begin{bmatrix}
	\dfrac{\sqrt{\sum_{i=1}^{p}\gamma^{2}_{i, p}}}{\gamma_{1, p}} & \cdots & 0\\
	\vdots                 & \ddots & \vdots\\
	0                      & \cdots & \dfrac{\sqrt{\sum_{i=1}^{p}\gamma^{2}_{i, p}}}{\gamma_{p, p}}
	\end{bmatrix}
	>\bf{I}.
	\end{equation}
	\end{comment}
	Thus we may conclude that the tests proposed by Dempster [8],[9], Bai and Saranadasa [2] are not optimal. 
	Similar arguments by taking ${\bf B}=[\mbox{tr}({\bf R}^{2})]^{-1/2}{\bf D}^{-1}_{0}$, where ${\mbSigma}={\bf D}^{\frac{1}{2}}_{0}{\bf R}{\bf D}^{\frac{1}{2}}_{0}$ with ${\bf D}_{0}=\mbox{diag}(\sigma_{11}, \ldots, \sigma_{pp})$, we may also conclude that tests used the information of diagonal elements of ${\bf S}$, 
	such as Srivastava and Du [23], Srivastava [24], Chen and Qin [6], Park and Ayyala [21]  are not optimal neither.
	
	\vspace{0.3cm}
	\indent {\bf Case 2.} Compare the tests constructed by using some correlations for the estimation of covariance
	matrix. Feng et al. [10] followed Bai and Saranadasa's model assumptions and improved the works of 
	Chen and Qin [6], Park and Ayyala [21] by adding correlations into consideration. They divided the $p$ variables
	into several small parts for invertible covariance matrix and then added those corresponding Hotelling $T^{2}$-test
	statistics up, which is called the composite $ T^{2}$ test. The asymptotic power function of the composite $T^{2}$
	test is of the form
	\vspace{-0.5cm}
	\begin{align}~\label{eq:4.1} %%p.21
		\beta_{CT}(\mbmu) \approx \Phi\left(-z_{\alpha}+\dfrac{n\mbmu^\top \mbSigma^{-1}_{\mathcal{O}^{K}}\mbmu}
		{\sqrt{2\mbox{tr}({\bf \Gamma}^{2}_{K})}}\right),
	\end{align}
	where ${\bf \Gamma}_{K} = \mbSigma^{1/2}\mbSigma^{-1}_{\mathcal{O}^{K}}\mbSigma^{1/2}$ 
	and $\mathcal{O}^{K}
	=\{ A^{0}_{1},\ldots, A^{0}_{N}\}$, for the details see Feng et al. [10] (p.1423). To avoid the asymptotic power 
	always being one as $p \to \infty$, some further conditions are needed. Note that under their assumption
	(C3): $\|{\mbmu}\|^{2}=O(n^{-1}p^{1/2})$, then equation (4.1) can be further reduced to that
	$\beta_{CT}(\mbmu)\approx\Phi(-z_{\alpha}+\frac{p^{1/2}{\mbdelta^\top \mbSigma^{-1}_{\mathcal{O}^{K}}\mbdelta}}
	{{\sqrt{2\mbox{tr}({\bf \Gamma}^{2}_{K})}}})$. We may see that the asymptotic power function 
	of composite test becomes
	$\beta_{CT}(\mbmu)=\Phi(-z_{\alpha}+(2d_1)^{-1/2}{\mbdelta^\top \mbSigma^{-1}_{\mathcal{O}^{K}}\mbdelta})$
	 if $\lim_{p\to \infty}\mbox{tr}({\mbGamma}^{2}_{K})/p =d_1$ holds. 
	But, note that $\mbSigma^{-1}_{\mathcal{O}^{K}}$ 
	will not be equal to $\mbSigma^{-1}$ generally. Feng et al. [10] basically made some assumptions on 
	the covariance matrix so that the estimator of covariance matrix having the block diagonal type matrix, 
	thus we may concern that the information may be lost in general.
	\begin{comment}
	\indent  For the comparisons based on the notion of ARE, the assumptions like those $\mbox{tr}({\mbSigma}^{4})
	=o(\mbox{tr}^{2}({\mbSigma}^{2}))$, $\mbox{tr}({\mbGamma}^{4}_{K})=o(\mbox{tr}^{2}({\mbGamma}^{2}_{K}))$ in the
	literature are necessary so that the asymptotic normal distribution of test statistics holds. As such, we need to
	check whether one of those conditions holds or not when making the asymptotic power comparisons.
	
	\indent Moreover, to perform the test statistic, generally, still we will face the basic difficulty of the estimation
	of population matrix, so does the matrix ${\mbGamma}^{2}_{K}$. Some more assumptions for covariance matrix $\mbSigma$
	are needed so that the consistent estimator $\widehat {{\mbox tr}({\mbGamma}^{2}_{K})}$ of
	${\mbox tr}{({\mbGamma}^{2}_{K})}$ exists.
	\end{comment}
	Theorem~3 tells us that the composite $T^{2}$ test of Feng et al. [10]
	is not optimal unless that ${\mbGamma}^{2}_{K}={\bf I}$, i.e., $\mbSigma^{-1}_{\mathcal{O}^{K}}=\mbSigma^{-1}$,
	which will not happen in their setup. Again, as in Case 1, we may conclude that there still exists room to
	develop test of more robust and powerful.
	
	\vspace{0.3cm} %%p.21
	\indent {\bf Case 3.} Compare the tests constructed by adopting the ridge regression type covariance estimator. 
	Chen et al. [5] imposed some regularizations on the sample covariance matrix and proposed a regularized Hotelling's
	$T^{2}$ statistic (RHT)
	\vskip-0.5cm
	\begin{equation}
		RHT(\lambda)= n\bar{\bf X}^\top \left({\bf S}+\lambda {\bf I}_{p}\right)^{-1}\bar{\bf X},
	\end{equation}
	where $\lambda > 0$. Note that the RHT statistic $n\bar{\bf X}^\top ({\bf S}+\lambda {\bf I}_{p})^{-1}\bar{\bf X}
	=n\bar{\bf X}^{\top}{\bf U}(\mbLambda+\lambda {\bf I})^{-1}{\bf U}^{\top}\bar{\bf X}$, which has the similar form as
	that of the decomposite $T^{2}_{N}$-test statistic. Note that $\lambda_{i, p}+\lambda$ is linear and
	$\lambda$ needs to be estimated. This is related to the Stein type shrinkage estimators. Their estimators of population
	eigenvalues may not be optimal. Ledoit and Wolf [13] studied the best linear estimator of the form
	$a\lambda_{i, p}+b, a, b >0, a+b=1$. Ledoit and Wolf [14] further claimed that the nonlinear estimators
	$\hat a^{or}_{i}$ are better than those of the best linear estimators $a\lambda_{i, p}+b, \forall i=1, \ldots, p$.
	It remains room to improve the estimators of eigenvalues. Ledoit and P{\'e}ch{\'e} [12] used the random
	matrix theory to claim that their nonlinear shrinkage eigenvalues estimator of the precision matrix ${\mbSigma}^{-1}$
	is optimal.
	
	\indent As noted in above, the ARE is based on the quantity of Kullback-Leibler divergence, 
	and the Stein loss function is proportional to the Kullback-Leibler divergence under 
	the multivariate normal setup. As such, the optimal orthogonally equivariance estimator 
	corresponds to the optimal power test. Among the class of orthogonally equivariant
	estimators, the decomposite $T^{2}_{N}$ test statistic digs out the optimal information of eigenvalues of 
	the precision matrix. Ledoit and Wolf [16] expected that their estimator in (2.8) 
	to be close to the inverse population matrix (precision matrix), and at the same time 
	its inverse can also be close to the population covariance matrix. 
	In comparisons with the tests mentioned above, the decomposite $T^{2}_{N}$-test is different from them. 
	We may expect that the decomposite $T^{2}_{N}$-test may perform better than 
	both the RHT proposed by Chen et al. [5] and the composite test proposed by Feng et al. [10]. 
	It is easy to note that the sample eigenvalues are not independent. 
	One of our main goals is to fulfill the hope that more information of population eigenvalues can be 
	digged out via the help of dedicated random matrix theory. 
	
	\subsection{Numerical power comparisons} %%p.22
	\indent Via Corollary~2, it is easy to see that the composite $T^{2}$-test of Feng et al. [10] has a similar form of 
	asymptotically local power function as that of the proposed decomposite $T^{2}_{N}$-test. Define the quantity $\dfrac{\mbdelta^{\top}\mbSigma^{-1}_{1}\mbdelta}{\sqrt{2d}} \bigg/ \dfrac{\mbdelta^{\top}\mbSigma^{-1}_{O^{\mathcal{K}}}\mbdelta}{\sqrt{2d_{1}}}$ 
	as the ARE of the decomposite $T^{2}_{N}$-test with respect to the composite $T^{2}$-test. 
	Note that, if the value of ARE is larger than 1, then the decomposite $T^{2}_{N}$-test has greater power than that of 
	the composite $T^{2}$-test.
	We make some simulation studies of power comparisons and AREs  for the decomposite 
	$T^{2}_{N}$-test and the composite $T^{2}$-test based on the intraclass correlation model. Namely, $\mbSigma=(\sigma_{ij})$,
	where $\sigma_{ij} = \rho^{|i-j|},~i=1,\ldots,p,~j=1\ldots,p$; $\rho\in(-1,1),\rho\neq 0$. Without loss of generality,
	we take $c=1/3$, the significance level $\alpha = 0.05$ and $\mbdelta=(\delta_{1},\ldots,\delta_{p})$ with $\delta_{i}\in (-1,,1).~i=1,\ldots,p$.
	When $p=20$, take $K = 2$, while $p=40$, take $K = 4$ in Table~1.
	\begin{table}[H]
		\centering
		\caption{Power comparison and ARE.}
		\resizebox{\textwidth}{!}{
			\begin{tabular}{rSSS|SSS}
				\toprule
				\multirow{2}{*}{} &
				\multicolumn{3}{c}{$p=20,K=2$ } &
				\multicolumn{3}{c}{$p=40,K=4$ } \\
				\midrule
				$\rho$ & {The decomposite} & {The composite}    & {ARE} & {The decomposite}  & {The composite} & {ARE} \\
				&{$T^{2}_{N}-test$}    &{$T^{2}$-test}      &             &{$T^{2}_{N}-test$}      &{$T^{2}$-test}     &              \\
				\midrule
				-0.2  & 0.3695  & 0.3480  &  1.0366  & 0.9615  & 0.9418   & 1.0560 \\
				0.2  & 0.7054  & 0.5893  &  1.1438   & 0.9025  & 0.8730  & 1.0501 \\
				-0.5  & 0.9073  & 0.6952  & 1.3293    & 0.9811   & 0.6447  & 1.7318 \\
				0.5  & 0.8251   & 0.2859  & 2.0758   & 0.9968  & 0.8162  & 1.6397 \\
				-0.8  & 0.9150   & 0.1878   & 3.1023    & 1.0000  & 0.6865  & 5.7759\\
				0.8  & 0.8848  & 0.0466  & 6.1533    & 1.0000 & 0.9759   & 5.4044\\
				\bottomrule
				\label{ARE}
			\end{tabular}
		}
	\end{table}
%\begin{comment}	
\vskip-5cm	
\section{Real data analysis}
\indent More than two decades have passed since the founding of the Taipei Rapid Transit Corporation (TRTC) in 1994.
Entering the 2.0 era, the Metro system is complete and is time for further expansion. 
A multi-point transferring model relieves congestion and disperses the current burden of existing transfer stations, 
therefore, providing the public with speedier and better transportation performance and quality. \\	
\indent In order to test whether there is a significant growth in population of public transportation, 
especially commuters  mainly take Taipei Metro System in recently years, we use data gathered from 1 July, 2015 to 30 April,
2020, including 108 stations' exit ridership on record. Since lacking the acknowledgment of distribution of $T^{2}_{N}$, 
we use bootstrap method to conduct the one sample testing problem with significant level $\alpha = 0.01$. 
\vskip-1cm
\subsection{The Bootstrap procedures for calculating \texorpdfstring{$T^{2}_{N}$}{T2N} are as follow:}
\begin{enumerate}\itemsep = -5pt
\item Calculate  column mean vector $\mXbar$ and sample covariance matrix ${\bf S}$ of  data set before resampling. Decomposite ${\bf S}$ into sample eigenvalues $\lambda_{i,p}, i=1,\ldots,p$ and its corresponding eigenvectors $u_{i}, i=1,\ldots,p$. 
\item Calculate ${\widehat \mbSigma}^{-1}_{LW}$ provided by Ledoit and Wolf [14] by using their algorithm of numerical implementation, the QuEST function in Ledoit and Wolf [15].
\item Then $T^{2}_{N}$ can be acquired as $T^{2}_{N} = n{\mXbar}^{\top}{\widehat \mbSigma}^{-1}_{LW}{\mXbar}$. 
\item Repeated random sample $95\%$ of the days from original data set with replacement, record the subset data each time.
\item Calculate  sample covariance matrix ${\bf S}$  and the corresponding $T^{2}_{N}$  for each collect data set. 
\end{enumerate}	
\indent After building up a sampling distribution by computing $T^{2}_{N}$ from 1000 times simulated data under the null hypothesis, we compare the test statistic before resampling to the sampling distribution. The empirical p-value is the proportion in the sampling distribution that are as extreme as the test statistics.\\
\indent We want to test whether there is a difference in mean ridership among stations under the following two cases. Let $\mbmu_{0}=(\mu_{1,0},\ldots,\mu_{108,0})^{\top}$ be the exit ridership mean vector of 108 stations, which is calculated from the second half year from July to December of 2015 as a comparison bench mark for mean testing. And our parameters to test, the exit ridership mean vector of 108 stations is denoted as $\mbmu = \left(\mu_{1},\ldots,\mu_{108}\right)^{\top}$.\\	

\subsection{The effect of Monthly Unlimited Transport Policy}
\indent In 2018, Mayors of Taipei and New Taipei City announced a new unlimited public transportation card, 
called the "All Pass Ticket", 
and is priced at NT$\$$1,280 $(\mbox{US}\$43.83)$ a month. It is released on April 16, 2018, 
and it is a periodical commuter ticket. It is valid for both buses and the Taipei Metro, and also for 
the first 30 minutes of a YouBike ride. Commuters across Taipei and New Taipei City are sure to benefit 
from the policy. Paying NT$\$$1,280 for 30 days unlimited rides works out to an average cost of NT$\$$42 per day.
Taipei Mayor Ke Wen-je said that as always, people are encouraged to use public transportation to help 
combat traffic congestion. On the other hand, New Taipei City Mayor Eric Chu said he hoped the new pass can 
help boost daily ridership in Taipei's public transportation system (March 12, 2018. Central News Agency). 
Hypothesis testing problem of interest is:
\vskip-2cm
\begin{align}~\label{eq:5.2}
H_{0}:{\mbmu_{1280}}={\mbmu}_{0}~\mbox{versus}~H_{1}:{\mbmu_{1280}}\ne {\mbmu}_{0}\;,
\end{align}
\vskip-0.5cm
\noindent where $\mbmu_{1280}$ is the mean vector of stations during the period of policy, 
and $\mbmu_{0}$ is the mean vector as defined before.\\
\indent Here we check the effectiveness of the ``All Pass Ticket" policy by bootstrap resampling process 
based on days from 2017 till 2019 to calculate the test statistic $T^{2}_{N}$, and the  
Hotelling's $T^{2}$-test statistic
\begin{comment}
The shadow area in figure \ref{Fig.1280} 
\end{comment}
shows that there are $4$ 
shuffled statistics out of 1000 less than the value, which is $5292.244$, of $T^{2}_{N}$-test statistic for 
the real data set. No matter what the significance level $\alpha$ is either 0.01 or 0.05, the empirical p-value 
is equal to $0.004$ which is less than $\alpha$. The value $5292.244$ is also less than $0.5\%$ quantile of 
the sampling distribution of value $5302.149$. Meanwhile, the result of using Hotelling's $T^{2}$-test while with empirical p-value being equal to $0$. It seems that there is a significant difference of the mean values in the aspect of exit ridership of each station during the monthly unlimited public transport card policy. \\
\indent For this real data set, by both the decomposite $T^{2}_{N}$-test with empirical p-value 0.04 and Hotelling's $T^{2}$-test with empirical p-value 0, H$_{0}$ in~(5.1) is rejected when the level of significance $\alpha$ is either 0.01 or 0.05. Note that no matter how small the significance level $\alpha$ is, H$_{0}$ is strongly rejected by Hotelling's T$^{2}$-test, with empirical p-value 0. This indicates that the decomposite $T^{2}_{N}$-test has the advantage over Hotelling's T$^{2}$-test by the bootstrap procedure in the analysis of this real data set.

\section{Conclusion and Future Study}
	\indent It is generally hard to compare tests well based on a single index, 
	for there are so many nuisance parameters when the dimension is large. 
	Some other statistical aspects are also needed to be incorporated into consideration 
	for the comparison of tests. $T^{2}_{N}$ 
	defined in (3.1) is constructed by the use of optimal estimators of eigenvalues of 
	the precision matrix as pointed out by Ledoit and Wolf [16]. 
	For there were no much work using these results from the data analysis point of views in the literature, 
	we adopt the permutation test based on good test statistics which may be easy to perform 
	and be robust in practice. Based on the discussions above, it seems reasonable 
	to adopt the decomposite $T^{2}_{N}$ statistic to perform the bootstrap procedure 
	for analyzing large dimensional data sets.\\
	\indent The rotation equivariance property is quite appropriate in the general situation where one has no 
	{\it prior} information about the orientation of the eigenvectors of population covariance matrix. 
	However, without having the consistent estimators of population eigenvalues matrix $\mbGamma$, 
	it is still difficult to perform the test statistic precisely well even under the null hypothesis $H_0$. 
	Those tests incorporated with the information of ${\bf S}$ existing 
	in the literature also face the same difficulty, such as the estimation of ${\mbGamma}^{2}_{K}$ Feng et al. [10]. 
	One of the main goals of this work is to find out more information about population eigenvalues with the 
	help of delicate random matrix theory. As we may note that the joint density function of those 
	dependent eigenvalues is well known for the Wishart ensemble, and it is given by the 
	Mar{\v c}enko-Pastur distribution for a system with large dimension when 
	$c \in (0, 1)$. So the statistical significance of the correlations in the large system can be obtained 
	from the empirical eigenvalue spectrum distribution of the sample covariance matrix via the 
	Mar{\v c}enko-Pastur distribution. This is one of the main advantages of the approach to obtain 
	the consistent eigenvalues and eigenvectors of population counterparts.
	
	\indent  If the matrix ${\bf D}$ is equal to the identity matrix, then our proposed $T_{N}^{2}$-test will be 
	optimal for  the hypothesis testing problem (eq:1.1). In this ideal situation, by Corollary~3 
	we then base on the normalized test statistic $(T^{2}_{N}-p)/\sqrt {2p}$ and usual normal theory to do the 
	work of data analysis. However, this study indicates that both the Stein's estimator (2.1) 
	and the Ledoit and Wolf's estimator (2.8) are not the
	consistent estimators of $\mbSigma$. For the application of principle of analysis, we may remark that 
	it is still open to find out the consistent estimators of population eigenvalues and eigenvectors of 
	$\mbSigma$ in the large dimensional system. At this stage, it may be too optimistic to expect 
	the whole information of $\mbSigma$ can be revealed without any {\it a prior}
	knowledge in its structure. Hence, we put this difficult but important problem as a future study.
%	\section*{Supplementary Materials}	
%	\noindent The online supplementary material contains the proof of Theorem~3.

\section*{Acknowledgments}%\\[0.3cm]
	\noindent The author would like to thank Professors Z.R. Chen and H.N. Hong from National Chiao Tung University, and Professor S.Y. Huang from Academia Sinica for their helpful discussions. 
	\begin{comment}
	The author would also like to thank the anonymous referees and an associate editor for their valuable comments that have led to an improved presentation of the paper. 
	\end{comment}

\section*{References}
\begin{enumerate}
	\item {T.W. Anderson (2003)}, { An Introduction to Multivariate Statistical
		Analysis}. 3rd edition. {\em Wiley, New York}.
	\item {Bai and Saranadasa (1996)}, {Effect of High Dimension: by an Example of a Two Sample Problem}.
	{\em Statistica Sinica}, Vol. 6, No.{\bf 2}, 311--329.		
	\item {Bai, Z.D. and Miao, B.Q. and Yao, J.F. (2003)},
	{Convergence rates of spectral distributions of large sample covariance matrices}.
	{\em SIAM J. Matrix Anal. Appl.}, Vol.25, 105--127.		
	\item {Bai, Z.D. and Miao, B.Q. and Pan, G.M. (2007)},  
	{On asymptotics of eigenvectors of large sample covariance matrix}.
	{\em Ann. Probab.}, Vol.35, 1532--1572.	   	
	\item {Chen, L.S. and Paul, D. and Prentice, R.L. and Wang, P. (2011)},	
	{A Regularized {H}otelling’s ${T}^{2}$-Test for Pathway Analysis in Proteomic Studies}.
	{\em J. Am. Stat. Assoc.}, Vol.106, No.{\bf 496}, 1345--1360.	
	\item{Chen, S.X. and Qin, Y.L. (2010)},	
	{A two-sample test for high-dimensional data with applications to gene-set testing}.
	{\em Ann. Statist.}, Vol.38({\bf 2}), 808--835.	
	\item {Choi, S.I. and Silverstein, J.W. (1995)},
	{Analysis of the limiting spectral distribution of large dimensional random matrices}.
	{\em J. Multivariate Anal.}, Vol.54({\bf 2}), 295--309.	
	\item {Dempster, A.P. (1958)},
	{A high dimensional two sample significance test}.
	{\em Ann. Math. Statist.}, Vol.29({\bf 4}), 995--1010.	
	\item{Dempster, A.P. (1960)},
	{A significance test for the separation of two highly multivariate small samples}.
	{\em Biometrics}, Vol.16({\bf 1}),41--50.	
	\item{Feng, L. and Zou, C. and Wang, Z. and Zhu, L. (2017)},
	{Composite ${T}^{2}$ test for high-dimensional data}.
	{\em Statistica Sinica}, Vol.27({\bf 3}), 1419--1436.	
	\item{Johnstone, I. M. and Paul, D. (2018)},
	{{P}{C}{A} in High Dimensions: An Orientation}.
	{\em Proc. IEEE}, Vol.106({\bf 8}), 1277--1292.   	
	\item{Ledoit, O. and P{\'e}ch{\'e}, S. (2011)},
	{Eigenvectors of some large sample covariance matrix ensembles}.
	{\em Probab. Theory Relat. Fields}, Vol.151, 233--264.	
	\item{Ledoit, O. and Wolf, M. (2004)},
	{A well-conditioned estimator for large-dimensional covariance matrices}.
	{\em J. Multivariate Anal.}, Vol.88, 365--411. 	
	\item{Ledoit, O. and Wolf, M. (2012)},
	{Nonlinear shrinkage estimation of  large-dimensional covariance matrices}.
	{\em Ann. Statist.}, Vol.40({\bf 2}), 1024--1060.
	\item{Ledoit, O. and Wolf, M. (2017)},
	{Numerical implementation of the QuEST function}.
	{\em Computational Statistics and Data Analysis}, Vol.115, 199--223.
	\item{Ledoit, O. and Wolf, M. (2018)},
	{Optimal estimation of a large-dimensional covariance matrix under Stein's loss}.
	{\em Bernoulli}, Vol.24({\bf 4B}), 3791--3832.    	
	\item{Li, H. and Aue, A. and Paul, D. and Peng, J. and Wang, P. (2020)},
	{An adaptable generalization of Hotelling’s $T^{2}$-test in high dimension}.
	{\em Ann. Statist.}, Vol.48({\bf 3}), 1815 -- 1847.	
	\item{Mar{\v{c}}enko, V.A. and Pastur, L.A. (1967)},
	{Distribution of eigenvalues for some sets of random}.
	{\em Sb. Math.}, Vol.1, 457--483.   	
	\item{Muirhead, R.J. (1982)},
	{Aspects of Multivariate Statistical Theory}.
	{\em Wiley, New York}. 	
	\item{Pan, G.M. and Zhou, W. (2011)},
	{Central limit theorem for {H}otelling's ${T}^{2}$ statistic under large dimension}.
	{\em Ann. Appl. Probab.}, Vol.21, 1860--1910.          	
	\item{Park, J. and Ayyala, D.N. (2013)},
	{A test for the mean vector in large dimension and small samples}.
	{\em J. Statist. Plan. Infer.}, Vol.143({\bf 5}), 929--943.	
	\item{Pitman, E.J.G. (1948)},
	{Lecture Notes on Nonparametric Statistical Inference: Lectures Given for the University of North Carolina}.
	University of North Carolina.	
	\item{Srivastava, M.S. and Du, M. (2008)},
	{A test for the mean vector with fewer observations than the dimension}.
	{\em J. Multivariate Anal.}, Vol.99({\bf 3}), 386--402.	
	\item{Srivastava, M.S. (2009)},
	{A test for the mean vector with fewer observations than the dimension under non-normality}.
	{\em J. Multivariate Anal.}, Vol.100({\bf 3}), 518--532.	
	\item{Silverstein, J.W. (1995)},
	{Strong convergence of the empirical distribution of eigenvalues of large dimensional random matrices}.
	{\em J. Multivariate Anal.}, Vol.55({\bf 2}), 331-339.	
	\item{Stein, C. (1975)},
	{Estimation of a covariance matrix}.
	Rietz lecture, 39th Annual Meeting IMS.	
	\item{Stein, C. (1986)},
	{Lectures on the theory of estimation of many parameters}.
	{\em J. Math. Sci}, Vol.34, 1373--1403.	
\end{enumerate}

	\vskip .65cm
	\noindent
	Institute of Statistical Science, Academia Sinica, Taipei.
	\vskip 2pt
	\noindent
	E-mail: teresatsai@stat.sinica.edu.tw
	\vskip 2pt
	
\end{document}